\documentclass[dmucheads]{dmucthesis}

\usepackage[portuguese]{babel}
\usepackage{url}
\usepackage{alltt}
\usepackage[utf8x]{inputenc} 
\usepackage[T1]{fontenc}
\usepackage{amssymb}
\usepackage[all,arc]{xy}
\usepackage{amsthm}
\usepackage{makeidx}
\usepackage{xcolor}
\usepackage{amsmath}
\usepackage{graphicx}
\usepackage{caption}
\usepackage{latexsym}
\usepackage{multicol}
\usepackage{amsmath}
\usepackage{tikz}
\usepackage{verbatim}
\usepackage{fancyhdr}
\usetikzlibrary{arrows,decorations.pathmorphing,backgrounds,fit,matrix}

\addto\captionsportuguese{%
}

\newtheorem{tma}{Teorema}[section]
\newtheorem{cor}[tma]{Corolário}
\newtheorem{pro}[tma]{Proposição}
\newtheorem{lem}[tma]{Lema}

\theoremstyle{definition}
\newtheorem{ex}{Exemplo}[section]
\newtheorem{cex}[ex]{Contra-exemplo}
\newtheorem{defc}[ex]{Definição}
\newtheorem{nota}[ex]{Nota}

\def\inv#1{#1^{-1}}

\def\nar{ \save \ar@{>->} \restore }
\def\rar{ \ar@{->>}}

\def\id{\mathrm{id}}

\tikzset{ar/.style={thick,->},nar/.style={thick,>->},rar/.style={thick,->>}}
\tikzset{descr/.style={fill=white,inner sep=2.5pt}}
\tikzset{bij/.style={above,sloped,inner sep=0.5pt}}

\def\mathbfdef#1{\expandafter\def\csname#1\endcsname{{\rm\bf#1}}}
\mathbfdef{Conj} \mathbfdef{Grp} \mathbfdef{Ab} \mathbfdef{Top} \mathbfdef{C} \mathbfdef{Eq(C)} \mathbfdef{GrpTop}
\def\EqC{{\rm\bf Eq(C)}}
\def\PtC{{\rm\bf Pt(C)}}
\def\Pt#1{\mathrm{\mathbf{Pt}_{#1}(\mathbf{C})\ }}

\DeclareMathOperator{\im}{im}
\DeclareMathOperator{\coig}{coig}
\DeclareMathOperator{\Alg}{\mathbf{Alg}}
\DeclareMathOperator{\Obj}{Obj}
\DeclareMathOperator{\Aut}{Aut}

\def\Teoria{\mathbb {T}}

\thesistitulo{VARIEDADES DE ÁLGEBRAS TOPOLÓGICAS}
\thesisautor{Lucas Taylor Earl}
\thesisgrau{Geometria, Álgebra e Análise}

\thesispresidente{Maria Manuela Oliveira de Sousa Antunes Sobral}
\thesisorientador{Maria Manuel Pinto Lopes Ribeiro Clementino}
\thesisvogalA{Gonçalo Gutierres da Conceição}

\thesisdatarealizacao{Data: 14 de Junho de 2013}

\thesisresumo{%
\pagestyle{plain}
As álgebras topológicas têm propriedades que estendem às dos grupos topológicos \cite{Borceux,Borceux2}, mas será que existem produtos semidirectos para álgebras topológicas, tal como no caso dos grupos? Primeiramente, expressam-se conceitos na linguagem de categorias que capturam as propriedades dos grupos. A seguir, observam-se resultados sobre grupos topológicos agora estendidos para variedades topológicas. No final, deduz-se que para álgebras topológicas que satisfazem determinados axiomas, existem produtos semidirectos e caracterizamo-los como em \cite{Inyangala}.
}

\thesisabstract{
Topological algebras have properties that extend naturally to those of topological groups \cite{Borceux,Borceux2}, but is it the case that semi-direct products exist as in the category of groups? Firstly, we express concepts in categorical language that capture group properties. We then observe the results about topological groups now extended to varieties of topological algebras. We conclude showing that for topological algebras obeying certain axioms, there exist semi-direct products which we caracterize like in \cite{Inyangala}.
}

\thesiskeywords{protomodular, semi-direct product, universal algebra, topology}

\thesispalavraschave{protomodular, produto semidirecto, álgebra, to\-po\-lo\-gia}

\begin{document}
\thispagestyle{empty}
\thesismaketitle
\thispagestyle{empty}
\pagestyle{plain}
\thesismakeabstract

\thesisagradecimentos{
\pagenumbering{roman}
À Professora Doutora Maria Manuel Pinto Lopes Ribeiro Clementino expresso o meu sincero agradecimento pelo apoio fundamental,  disponibilidade na partilha do saber e pelos seus valiosos conselhos. Acima de tudo, obrigado por estimular o meu interesse pelo estudo da matemática.
\\

Também queria agradecer ao Professor Doutor Andrea Montoli pela sua ajuda no estudo dos monomorfismos normais e no esclarecimento de dúvidas ao longo desta jornada.
\\

Ao Departamento de Matemática da Universidade de Coimbra, agra\-de\-ço as condições disponibilizadas para a realização deste trabalho.
\\

À minha família, pelos valores morais que sempre me transmitiram e por me ajudarem a crescer pessoal e intelectualmente. Pelo inestimável apoio apesar da distância, e pela contínua paciência e compreensão.
\\

À Mara Sofia da Cruz Antunes, à Teresa Sousa e ao Jason Nobre Bolito pelas suas sugestões linguísticas e pela amizade profunda prestada.  
\\

A todos os demais...
}
\tableofcontents

\setcounter{chapter}{-1}
\chapter{Introdução}
Álgebra e topologia são dois ramos desenvolvidos na matemática com abordagens diferentes. A topologia explora as questões de conectividade, compactidão, etc., de espaços, enquanto que na álgebra se trabalha com estruturas de elementos que satisfazem axiomas equacionais. Essas estruturas são aplicadas em várias partes da mateḿatica aplicada, inclusive à física teórica onde representam partículas, campos, ondas, interacções, etc. Uma questão pertinente que surge é se se pode definir objectos que são simultaneamente algébricos e topológicos.

Em 1891, Sophus Lie abriu o caminho introduzindo os seus «grupos contínuos» que são munidos da estrutura de uma variedade diferenciável, que em particular, é um espaço topológico \cite{Lie}. Inicialmente os grupos de Lie foram aplicados às equações diferenciáveis e aos grupos de transformações, porém, estes têm aplicações importantes em vários ramos matemáticos e, em especial, na matemática aplicada. No entanto, obter-se-ia maior generalidade nas estruturas implementadas. Interessava, por exemplo, estudar os espaços topológicos sobre uma álgebra específica.

Grupos topológicos, primeiramente estudados por Schreier em 1925 \cite{Schreier}, são um exemplo do poder de uma álgebra topológica, ou seja, um grupo munido de uma topologia «compatível» com este. Os resultados que deles se obtêm são fantásticos. Começando com Van Dantzig \cite{VanDantzig} (que foi o primeiro a usar o termo {\it grupos topológicos}) outras estruturas simultaneamente algébricas e topológicas foram inspeccionadas: anéis, corpos, módulos.

Com o aparecimento da teoria das categorias, foi natural considerar uma álgebra sobre uma teoria algébrica qualquer. Observou-se que para uma teoria qualquer o functor de esquecimento preserva limites e colimites. A caracterização de álgebras protomodulares levou Borceux e Clementino a um método sistemático para provar resultados clássicos sobre essas álgebras \cite{Borceux3}.

\begin{center}* * *\end{center}

No capítulo 1 apresentam-se ambientes apropriados dentro da teoria das categorias para o estudo da álgebra. Consideramos os conceitos de grupos tais como lemas homológicos, quocientes e relações nesse prisma. Comparamos este ambiente às categorias abelianas mostrando que as semi-abelianas capturam propriedades de grupos, enquanto que as abelianas representam os grupos abelianos.

No capítulo 2, começa-se por expor os grupos topológicos, verificando que se encaixam no ambiente das categorias homológicas, e portanto, satisfazem os lemas de homologia. Demonstram-se também os resultados conhecidos dos grupos topológicos e introduz-se a noção que estes formam uma variedade de álgebras topológicas. Conclui-se mostrando que os resultados para $\GrpTop$ se estendem para as variedades de álgebras topológicas.

O capítulo 3 inicia-se apresentando os produtos semidirectos em $\Grp$. Em seguida, mostra-se que existe uma generalização desta noção para categorias protomodulares que equivale no caso dos grupos. No final demonstra-se que existem produtos semidirectos em álgebras topológicas sobre teorias que satisfazem certos axiomas.

\chapter{Categorias semi-abelianas}
As categorias de módulos sobre um anel têm propriedades bem identificadas, que há mais de 50 anos levaram ao conceito de categoria abeliana \cite{MacLane2,MacLane}. No entanto, a categoria $\Grp$ dos grupos e homomorfismos não é  abeliana e surge uma questão natural de como se podem capturar as propriedades essenciais dos grupos no am\-biente das categorias. Durante décadas, investigadores dedicados foram desenvolvendo aos poucos conceitos que representam vários comportamentos de $\Grp.$ Este estudo culminou com a noção de categoria semi-abeliana formulada por Janelidze, Márki e Tholen  \cite{Janelidze} que é a junção da noção de categoria exacta de Barr com a de categoria protomodular de Bourn \cite{Barr,Bourn}.

\begin{figure}[h] \centering
\begin{tikzpicture}[parent anchor=south,child anchor=north,grow=south,sibling distance=2.6cm, level distance=3cm,
nt/.style={rectangle, minimum size=6mm, 
very thick, draw=green!50!black!50, top color=white, bottom color=red!50!black!20, font=\itshape
}, ntt/.style={rectangle, minimum size=6mm, 
very thick, draw=green!60!black!40, top color=white, bottom color=green!40!black!30, font=\itshape
}]

\node[nt]{abeliana}[parent anchor=east, child anchor=west,grow=east, edge from parent/.style={-,thick,draw}]
	child {node[nt] {aditiva}}
	child {node[ntt] {semi-abeliana}
		child{node[nt] (e) {exacta} child{node[nt]{efectiva}} child {coordinate(b)}}
		child{node[nt] (h) {homológica} child {node[nt] (r){regular}} child{node[nt] {protomodular}} child {node[nt] {pontuada}}}};

\path[-,very thick,draw,white] (e.east) edge (b);
\path[-,thick,draw] (e) edge (r);
\path[-,thick,draw] (h.east) edge (r.west);

\end{tikzpicture}
\caption{}
\end{figure}

\section{Relações de equivalência internas}
Em grupos tal como em variedades universais não se pode menosprezar o conceito de relações de equivalência (denominadas congruências em álgebra universal). Os teoremas de isomorfismo de Emmy N\"other são teoremas sobre essas relações de equivalência na categoria dos grupos. Na secção \ref{cat_hom} é apresentado um ambiente mais geral em que esses teoremas homológicos ainda são válidos--para uma descrição conceptual remete-se o leitor para \cite{Clementino}.

Com esse objectivo, introduz-se a noção interna de relação de equivalência. Para isso, em primeiro lugar caracterizam-se, em linguagem de categorias, as relações de equivalência na categoria $\Conj$ dos conjuntos e funções. Em \Conj, uma relação do conjunto $X$ no conjunto $Y$ é um subconjunto do seu produto: $R\subseteq X\times Y$. Para as relações sobre o mesmo conjunto, $R\subseteq X\times X,$ pode definir-se uma relação de equivalência, i.e., uma relação que é reflexiva, simétrica e transitiva. No âmbito da teoria das categorias, examinam-se as propriedades de uma relação de equivalência visando as suas projecções. 

Em \Conj, a relação $\sim$ será reflexiva se cada elemento for em relação a si próprio, isto é, a relação $R\rightrightarrows X$ for reflexiva em $\Conj$ se, para todo o $x \in X, \ (x,x)\in R,$ ou seja, se a diagonal de $X$, $\Delta_X=\{(x,x) | x\in X\}$, for um subconjunto de $R.$ Pode dizer-se ainda que o morfismo $(1_X, 1_X)$ se factoriza através de $R,$ como mostra o diagrama seguinte:
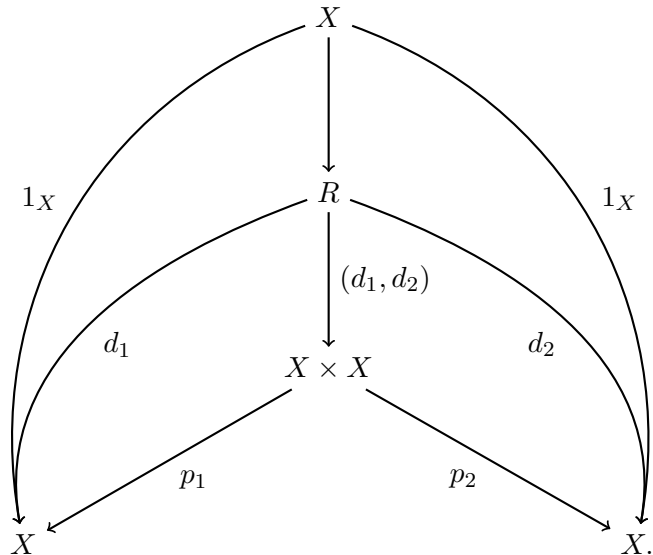
\[
\begin{tikzpicture}
\matrix (m) [matrix of math nodes, row sep=1.8cm, column sep=3cm]
{&X&\\&R&\\&X\times X & \\ X& & X.\\};
\path[ar]
(m-1-2) edge (m-2-2)
(m-2-2) edge node[auto]{$(d_1,d_2)$} (m-3-2)
(m-3-2) edge node[auto]{$p_1$} (m-4-1)
(m-3-2) edge node[auto,swap]{$p_2$} (m-4-3);
\path[ar,out=200,in=100] (m-1-2) edge node[auto,swap]{$1_X$} (m-4-1);
\path[ar,out=-20,in=80] (m-1-2) edge node[auto]{$1_X$} (m-4-3);
\path[ar,out=200,in=100] (m-2-2) edge node[auto]{$d_1$} (m-4-1);
\path[ar,out=-20,in=80] (m-2-2) edge node[auto,swap]{$d_2$} (m-4-3);
\end{tikzpicture}
\]

Para ser simétrica, $R$ satisfará: $$(x,y)\in R \Leftrightarrow (y,x) \in R,$$ isto é, se permutar o papel da primeira e segunda componentes a relação manter-se-á. Se $R$ for simétrica, seja
$\sigma:R\to R;\ (x,y)\to (y,x)$ a função que permuta os papéis das suas componentes. Devem verificar-se as identidades seguintes: $d_2 \circ \sigma = d_1$ e $d_1 \circ \sigma = d_2.$

Transitividade em \Conj\ significa a implicação: $(x,y), (y,z) \in R \Rightarrow (x,z) \in R.$ Pretende escrever-se esta propriedade também em termos de morfismos. Seja $S = \{(x,y,z) | (x,y), (y,z) \in R\};$ garantidamente existem as projecções: $(x,y,z)\mathop{\mapsto}\limits^{q_2} (x,y);$ e $(x,y,z)\mathop{\mapsto}\limits^{q_1} (y,z),$ e se $R$ for transitiva define-se $\tilde{q}:S\to R;\ (x,y,z)\mapsto (x,z)$ onde todos estes devem satisfazer o diagrama a seguir:

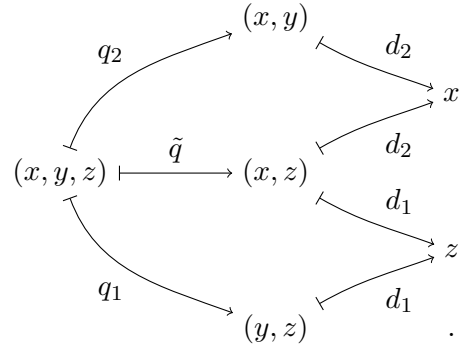
\[
\begin{tikzpicture}
\matrix (m) [matrix of math nodes, column sep=1.5cm, row sep=.5cm]{& (x,y) & \\ && x\\ (x,y,z) &(x,z)&\\ && z \\ & (y,z)&.\\};
\path[|->][out=70,in=200] (m-3-1) edge node[auto]{$q_2$} (m-1-2);
\path[|->] (m-3-1) edge node[auto]{$\tilde{q}$} (m-3-2);
\path[|->][out=-70,in=-200] (m-3-1) edge node[auto,swap]{$q_1$} (m-5-2);

\path[|->][out=30,in=200] (m-3-2) edge node[auto,swap]{$d_2$} (m-2-3);
\path[|->][out=-30,in=-200] (m-3-2) edge node[auto]{$d_1$} (m-4-3);

\path[|->][out=30,in=200] (m-5-2) edge node[auto,swap]{$d_1$} (m-4-3);
\path[|->][out=-30,in=-200] (m-1-2) edge node[auto]{$d_2$} (m-2-3);
\end{tikzpicture}
\]
Pode averiguar-se que o terno $(S,q_1,q_2)$ assim definido é o produto fibrado de $\begin{tikzpicture}
\matrix (m) [matrix of math nodes,column sep=1cm,row sep=1cm]
{R & X & R.\\};
\path[ar] (m-1-1) edge node[auto]{$d_1$} (m-1-2) 
(m-1-3) edge node[auto,swap]{$d_2$} (m-1-2);
\end{tikzpicture}$

Com estas propriedades, define-se a noção de uma relação de equivalência interna numa categoria qualquer com produtos fibrados, generalizando a noção em $\Conj$.

\begin{defc}
    Uma relação $d_1,d_2:R\rightrightarrows X$ numa categoria \C\ com produtos fibrados diz-se:
    \begin{itemize}
        \item Reflexiva se $(1_X,1_X)$ se factorizar através de $R;$
        \item Simétrica se existir $\sigma: R \to R$ que satisfaça $d_1 \circ \sigma = d_2,$ e $d_2 \circ \sigma = d_1;$
        \item Transitiva se no produto fibrado $(R\times_X R,q_1,q_2)$:

\[
\begin{tikzpicture}
\matrix (m) [matrix of math nodes,column sep=1cm,row sep=1cm]
{R\times_X R & R\\ R & X,\\};
\path[ar]
(m-1-1) edge node[auto]{$q_2$} (m-1-2)
(m-1-1) edge node[auto,swap]{$q_1$} (m-2-1)
(m-2-1) edge node[auto,swap]{$d_1$} (m-2-2)
(m-1-2) edge node[auto]{$d_2$} (m-2-2);
\end{tikzpicture}
\]
existir um morfismo $\tilde{q}:R\times_X R\to R$ tal que $d_2 \circ q_1 = d_2 \circ \tilde q$ e $d_1 \circ q_2 = d_1 \circ \tilde q.$
    \end{itemize}
    Uma relação numa categoria diz-se de equivalência se for reflexiva, simétrica e transitiva.
\end{defc}

As relações de equivalência internas dentro de uma categoria formam uma categoria \EqC, que tem como objectos as relações de equivalência e como morfismos diagramas comutativos da forma:

\begin{center}
\begin{tikzpicture}
\matrix (m) [matrix of math nodes,column sep=1.7cm,row sep=1.7cm]
{R & R'\\ X\times X & Y\times Y\\};
\path[ar]
(m-1-1) edge node[auto]{$\tilde{f}$} (m-1-2)
([xshift=-.5mm]m-1-1.south) edge node[auto,swap]{$d_1$} ([xshift=-.5mm]m-2-1.north)
([xshift=-.5mm]m-1-2.south) edge node[auto,swap]{$d_1'$} ([xshift=-.5mm]m-2-2.north)
([xshift=.5mm]m-1-1.south) edge node[auto]{$d_2$}([xshift=.5mm]m-2-1.north)
([xshift=.5mm]m-1-2.south) edge node[auto]{$d_2'$} ([xshift=.5mm]m-2-2.north)
(m-2-1) edge node[auto,swap]{$(f,f)$} (m-2-2);
\end{tikzpicture}
\end{center}
onde $f:X\to Y$ é um morfismo em \C.


Na matéria que irá ser abordada ainda neste capítulo usaremos relações de equivalência especiais. Anunciemos a definição do par núcleo de um morfismo.

\begin{defc}
O par núcleo de um morfismo $f:X\to Y$ é o par das projecções no produto fibrado:
\[
\begin{tikzpicture}
\matrix (m) [matrix of math nodes, column sep=1.3cm, row sep=1.3cm] 
{R[f] & X\\X & Y.\\};
\path[ar]
(m-1-1) edge node[auto]{$\pi_1$} (m-1-2)
(m-1-2) edge node[auto]{$f$} (m-2-2)
(m-1-1) edge node[auto,swap]{$\pi_2$} (m-2-1)
(m-2-1) edge node[auto,swap]{$f$} (m-2-2);
\end{tikzpicture}
\]Usa-se a notação $R[f]$ para designar $X\times_Y X.$
\end{defc}

\begin{pro}
Para qualquer morfismo $f:X\to Y$ em \C , o par núcleo \[
\begin{tikzpicture}
\matrix (m) [matrix of math nodes, column sep=1.3cm, row sep=1.3cm,text height=1.5ex,text depth=.25ex] 
{R[f] & X\\};
\path[ar] ([yshift=-.5mm]m-1-1.east) edge node[auto,swap]{$\pi_2$} ([yshift=-.5mm]m-1-2.west)
([yshift=.5mm]m-1-1.east) edge node[auto]{$\pi_1$} ([yshift=.5mm]m-1-2.west);
\end{tikzpicture}
\]
é uma relação de equivalência.
\end{pro}

\begin{proof}
O facto de a relação $R[f]$ ser reflexiva e simétrica vem da propriedade universal, como é visível nos diagramas abaixo:
\[
\begin{array}{cc}
\begin{tikzpicture}
\matrix (m) [matrix of math nodes, column sep=1.3cm,row sep=1.3cm]{X &&\\& R[f] & X\\ & X &Y\\};
\path[ar,dashed] (m-1-1) edge (m-2-2);
\path[ar] (m-2-2) edge node[auto]{$\pi_2$} (m-2-3)
(m-2-2) edge node[auto,swap]{$\pi_1$} (m-3-2)
(m-3-2) edge node[auto,swap]{$f$} (m-3-3)
(m-2-3) edge node[auto]{$f$} (m-3-3);
\path[ar][out=0,in=120] (m-1-1) edge node[auto]{$1_X$} (m-2-3)
[out=270,in=150] (m-1-1) edge node[auto,swap]{$1_X$} (m-3-2);
\end{tikzpicture}

& 
\begin{tikzpicture}
\matrix (m) [matrix of math nodes, column sep=1.3cm,row sep=1.3cm]{R[f] &&\\& R[f] & X\\ & X &Y\\};
\path[ar,dashed] (m-1-1) edge node[auto]{$\sigma$} (m-2-2);
\path[ar] (m-2-2) edge node[auto]{$\pi_2$} (m-2-3)
(m-2-2) edge node[auto,swap]{$\pi_1$} (m-3-2)
(m-3-2) edge node[auto,swap]{$f$} (m-3-3)
(m-2-3) edge node[auto]{$f$} (m-3-3);
\path[ar][out=0,in=120] (m-1-1) edge node[auto]{$\pi_1$} (m-2-3)
[out=270,in=150] (m-1-1) edge node[auto,swap]{$\pi_2$} (m-3-2);
\end{tikzpicture}\\
\end{array}
\]
Formando o produto fibrado de $R[f]$: 

\[\begin{tikzpicture}
\matrix (m) [matrix of math nodes,column sep=1.3cm,row sep=1.3cm]
{R[f]\times_XR[f] & R[f]\\ R[f] & X,\\};
\path[ar]
(m-1-1) edge node[auto]{$\xi_1$} (m-1-2)
(m-1-1) edge node[auto,swap]{$\xi_2$} (m-2-1)
(m-2-1) edge node[auto,swap]{$\pi_2$} (m-2-2)
(m-1-2) edge node[auto]{$\pi_1$} (m-2-2);
\end{tikzpicture}\] 
e aplicando a propriedade universal comprova-se que $R[f]$ é transitiva, como o diagrama a seguir indica:
\[
\begin{tikzpicture}
\matrix (m) [matrix of math nodes, column sep=1.3cm,row sep=1.3cm]{R[f]\times_X R[f] &&\\& R[f] & X\\ & X &Y.\\};
\path[ar,dashed] (m-1-1) edge node[auto]{$\tilde\xi$} (m-2-2);
\path[ar] (m-2-2) edge node[auto]{$\pi_2$} (m-2-3)
(m-2-2) edge node[auto,swap]{$\pi_1$} (m-3-2)
(m-3-2) edge node[auto,swap]{$f$} (m-3-3)
(m-2-3) edge node[auto]{$f$} (m-3-3);
\path[ar][out=0,in=120] (m-1-1) edge node[auto]{$\pi_2\circ \xi_1$} (m-2-3)
[out=270,in=150] (m-1-1) edge node[auto,swap]{$\pi_1\circ\xi_2$} (m-3-2);
\end{tikzpicture}
\]
\end{proof}

Às relações de equivalência que sejam pares núcleos chamam-se efectivas. Nem toda a relação de equivalência é efectiva. Isto acontece, por exemplo, na categoria dos espaços topológicos e nas aplicações contínuas, \Top.

\begin{ex}
Construir, para cada relação de equivalência, um morfismo que mostra que esta é efectiva é um procedimento directo em muitas categorias como $\Conj.$ Seja $d_1,d_2: R \rightrightarrows X$ uma equivalência sobre o conjunto $X.$ Descreve-se a projecção $\pi:X\to X/_\sim$ entre $X$ e o conjunto $X/_\sim$ das classes de equivalência, o que leva cada elemento de $X$ para a sua classe. Deste modo $R[\pi]$ é isomorfa a $R.$ \qed
\end{ex}

\section{Categorias exactas}
Categorias exactas--definidas pela primeira vez em \cite{Barr}--são categorias onde existe um sistema «óptimo»\ de factorização (no caso das categorias abelianas \(f = \mathrm{im}f \circ \mathrm{coim} f\)) e em que as relações de equivalência internas são sempre pares núcleos, podendo ser escritas em termos de aplicações quo\-cientes.

Nem todo o homomorfismo de grupos se factoriza através de um núcleo após um conúcleo como acontece no caso de categorias abelianas. Considere-se um homomorfismo entre grupos $f: X \to Y.$ Se $f$ for um homomorfismo entre grupos abelianos, a imagem de $f$ é um subgrupo normal de $Y,$ mas para grupos em geral, nem sempre se confirma. Se $f(X)$ não for normal, então não se pode definir $Y/f(X)$ e, por esse motivo, não é possível factorizar $f$ através da sua imagem.

No entanto, existe uma factorização mais fraca em $\Grp.$ Forma-se o grupo quociente, $X/f$, que consiste nas classes de equivalência na relação gerada por $f$ (i.e., $x\sim x' \Leftrightarrow f(x) = f(x')$). Escreve-se então, $f=m\circ e$ onde $m:X/f\to Y;\ [x]\mapsto f(x)$ e $e:X \to X/f;\ x\mapsto [x]$. Esses morfismos estão bem definidos como se pode comprovar. Além disso, $m$ é injectivo e portanto um monomorfismo. O morfismo $e:X\to X/f$ é um co-igualizador de $(R[f],\pi_1,\pi_2).$ Porém, esta factorização (de monomorfismo após epimorfismo regular em $\Grp$) deriva de um co-igualizador, razão pela qual é única a menos de isomorfismo.


Nesta secção identificam-se, inicialmente, propriedades de categorias que serão exactamente aquelas com este tipo de factorização única.

\begin{defc}
Uma categoria finitamente completa diz-se regular se:
\begin{enumerate}
    \item Tiver co-igualizadores de pares núcleos;
    \item Os seus epimorfismos regulares forem estáveis para produtos fibrados.
\end{enumerate}
\end{defc}

De modo a verificar a utilidade desta definição provaremos que numa categoria regular qualquer morfismo tem uma factorização canónica tal como no caso dos grupos.

\begin{pro}
\label{factor}
Numa categoria regular, qualquer morfismo $f:X\to Y$ em \C\ tem uma factorização única (a menos de isomorfismo) \(f=m\circ e,\) onde $e$ é um epimorfismo regular e $m$ é um monomorfismo.
\end{pro}

\begin{proof}
Considere-se o par núcleo de $f,$
\[
\begin{tikzpicture}[text height=1.5ex,text depth=.25ex]

\matrix (m) [matrix of math nodes, column sep=1.5cm,row sep=1.5cm]{R[f] & X & Y.\\};
\path[ar]
(m-1-2) edge node[auto] {$f$} (m-1-3)
([yshift=-.5mm]m-1-1.east) edge node[auto,swap]{$\pi_1$} ([yshift=-.5mm]m-1-2.west)
([yshift=.5mm]m-1-1.east) edge node[auto]{$\pi_2$} ([yshift=.5mm]m-1-2.west);
\end{tikzpicture}
\]
Seja $e:X\to Q$ o co-igualizador das projecções $\pi_1$ e $\pi_2$, então pela propriedade universal do co-igualizador existe um único morfismo $m$ como é evidente no diagrama a seguir:
\[
\begin{tikzpicture}[text height=1.5ex,text depth=.25ex]
\matrix (m) [matrix of math nodes, column sep=1.8cm,row sep=1.8cm]{R[f] & X && Y\\
&&Q&.\\};
\path[ar]
(m-1-2) edge node[auto] {$f$} (m-1-4)
(m-1-2) edge node[auto,swap]{$e$} (m-2-3)
([yshift=-.5mm]m-1-1.east) edge ([yshift=-.5mm]m-1-2.west)
([yshift=.5mm]m-1-1.east) edge ([yshift=.5mm]m-1-2.west);
\path[ar,dashed](m-2-3) edge node[auto,swap]{$m$} (m-1-4);
\end{tikzpicture}
\]
O facto de $m$ ser um monomorfismo vem do corolário \ref{cor} e $R[f]\cong R[e].$
\end{proof}

\begin{defc}
Na demonstração da proposição anterior, o morfismo $m:Q\to Y$ que é designado por $\im(f)$, é chamado de imagem de $f.$
\end{defc}

\begin{cor}
    Numa categoria regular os epimorfismos regulares coincidem com os fortes.
\end{cor}

Salienta-se mais uma propriedade útil das factorizações em categorias regulares.

\begin{pro}
\label{factorim}
Dado um quadrado comutativo numa categoria regular:
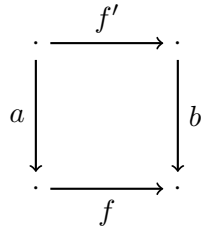
\[
\begin{tikzpicture}
\matrix (m) [matrix of math nodes, column sep=1.5cm, row sep=1.5cm]{\cdot & \cdot \\ \cdot & \cdot \\};
\path[ar] (m-1-1) edge node[auto]{$f'$} (m-1-2)
(m-1-1) edge node[auto,swap]{$a$} (m-2-1)
(m-2-1) edge node[auto,swap]{$f$} (m-2-2)
(m-1-2) edge node[auto]{$b$} (m-2-2);
\end{tikzpicture}
\]
existe uma factorização através das imagens, isto é, existe um único morfismo $\theta$ que torna o diagrama seguinte comutativo:
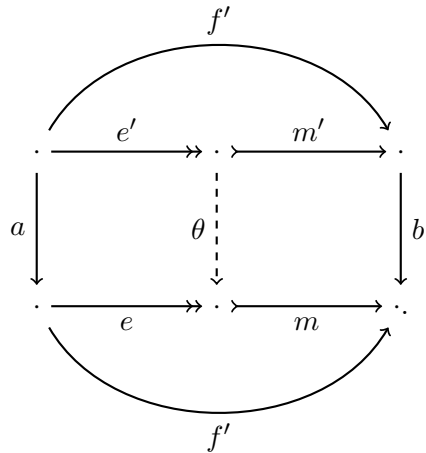
\[
\begin{tikzpicture}
\matrix (m) [matrix of math nodes, column sep=2cm, row sep=1.5cm, text height=1.5ex, text depth=0.25ex]{\cdot & \cdot & \cdot \\ \cdot & \cdot & \cdot .\\};
\path[ar] 
(m-1-1) edge node[auto,swap]{$a$} (m-2-1)
(m-1-3) edge node[auto]{$b$} (m-2-3);
\path[rar](m-2-1) edge node[auto,swap]{$e$} (m-2-2)
(m-1-1) edge node[auto]{$e'$} (m-1-2);
\path[nar] (m-1-2) edge node[auto]{$m'$} (m-1-3)
(m-2-2) edge node[auto,swap]{$m$} (m-2-3);
\path[ar,dashed] (m-1-2) edge node[auto,swap]{$\theta$} (m-2-2);
\path[ar] (m-2-1) edge[out=-60,in=-120] node[auto,swap]{$f'$} (m-2-3)
(m-1-1) edge[out=60,in=120] node[auto]{$f'$} (m-1-3);
\end{tikzpicture}
\]
\end{pro}

\begin{proof}
Seja $(\alpha_1,\alpha_2)$ o par núcleo de $f',$ e como estamos numa categoria regular o par núcleo de $f'$ é exactamente o par núcleo de $e'.$ De $b\circ f = m \circ e \circ a,$ decorre $$m\circ e \circ a \circ\alpha_1 = m\circ e \circ a \circ \alpha_2.$$ Todavia, como $m$ é um monomorfismo obtém-se $e\circ a\circ\alpha_1 = e\circ a\circ\alpha_2$ e então $e\circ a$ factoriza-se através de $e'$:
\[
\begin{tikzpicture}
\matrix (m) [matrix of math nodes, column sep=1.5cm, row sep=1.5cm, text height=1.5ex, text depth=0.25ex]{I & \cdot & \cdot & \cdot \\& \cdot & \cdot & \cdot .\\};
\path[ar] 
(m-1-2) edge node[auto,swap]{$a$} (m-2-2)
(m-1-4) edge node[auto]{$b$} (m-2-4);
\path[rar](m-2-2) edge node[auto,swap]{$e$} (m-2-3)
(m-1-2) edge node[auto]{$e'$} (m-1-3);
\path[nar] (m-1-3) edge node[auto]{$m$} (m-1-4)
(m-2-3) edge node[auto,swap]{$m$} (m-2-4);
\path[ar] ([yshift=-.4mm] m-1-1.east) edge node[auto,swap]{$\alpha_2$} ([yshift=-.4mm] m-1-2.west)
([yshift=.4mm] m-1-1.east) edge node[auto]{$\alpha_1$} ([yshift=.4mm] m-1-2.west);
\path[ar,dashed] (m-1-3) edge (m-2-3);
\end{tikzpicture}
\]
\end{proof}

Adiciona-se uma condição à definição de regular para definir uma categoria exacta de Barr.

\begin{defc}
Uma categoria \C\ é \emph{exacta} se for regular e \emph{efectiva} (toda a sua relação de equivalência interna é efectiva).
\end{defc}


Nesta altura, afirma-se que a categoria dos conjuntos é exacta. 
\begin{pro}
Apesar de ter co-igualizadores quaisquer, a categoria dos espaços topológicos não é regular
.
\end{pro}

\begin{proof}
Epimorfismos em $\Top$ nem sempre são estáveis para produtos fibrados. Para ver isso basta um contra-exemplo:
Definamos conjuntos subjacentes
\[A=\{a_1,a_2,a_3,a_4\},\ B=\{b_1,b_2,b_3\}\text{\  e\  } C=\{c_1,c_2,c_3\};\]
e funções
\[f:A\to C,\text{\  e \ } g:B\to C\] com $a_1,b_1\mapsto c_1,\ a_2,a_3\mapsto c_2,$ \hbox{$a_4,b_2,b_3\mapsto c_3.$} Se $\{a_1,a_2\}$ e $\{b_1,b_3\}$ pertencerem às topologias de $A$ e $B$ respectivamente e se a topologia de $C$ for indiscreta, confirma-se de maneira directa que $f$ é uma aplicação quociente mas no produto fibrado:
\[
\begin{tikzpicture}
\matrix (m) [matrix of math nodes, row sep=4em,
column sep=3.4em, text height=1.5ex, text depth=0.25ex]
{A\times_C B & B \\ A & C\\};
\path[ar] (m-1-1) edge node[auto]{$\pi_2$} (m-1-2)
(m-1-1) edge node[auto,swap]{$\pi_1$} (m-2-1)
(m-1-2) edge node[auto]{$g$} (m-2-2)
(m-2-1) edge node[auto,swap]{$f$} (m-2-2);
\end{tikzpicture}
\] $\pi_2$ não é um quociente, pois $\{(a_1,b_1)\} = \pi_1^{-1}(\{a_1,a_2\})$ é aberto mas $\{b_1\}$ não é.
\end{proof}

\section{Categorias protomodulares}
A noção de sucessão exacta é uma noção importante no estudo da homologia. Considera-se nesta secção um teorema de homologia bem conhecido, o lema dos cinco, que diz: 
\begin{tma}{Lema dos Cinco\\}
Num diagrama comutativo em $\Grp$:
\[
\begin{tikzpicture}
\matrix (m) [matrix of math nodes, row sep=1.3cm, column sep=1.3cm, text height=1.5ex, text depth=0.25ex]{K[f] & X & Y\\ K[f'] & X' & Y'\\};
\path[ar] 
(m-1-1) edge node[auto,swap] {$a$} (m-2-1)
(m-1-2) edge node[auto] {$b$} (m-2-2)
(m-1-3) edge node[auto] {$c$} (m-2-3);

\path[rar]
(m-2-2) edge node[auto,swap] {$f'$} (m-2-3)
(m-1-2) edge node[auto] {$f$} (m-1-3);
\path[nar]
(m-1-1) edge node[auto] {$\ker f$} (m-1-2)
(m-2-1) edge node[auto,swap] {$\ker f'$} (m-2-2);
\end{tikzpicture}
\]sendo $f$ e $f'$ epimorfismos regulares, se $a$ e $c$ forem isomorfismos, então $b$ é um isomorfismo.
\end{tma}

O objectivo principal desta secção é caracterizar as categorias $\C$ que satisfazem um teorema mais fraco, o lema cindido dos cinco que, para além das hipóteses do lema dos cinco requer que $f$ e $f'$ sejam epimorfismos cindidos. Inicialmente definem-se subcategorias da categoria $\C\downarrow B $ sobre um objecto $B$ de $\C$, que chamaremos a categoria dos pontos sobre $B$. Seguidamente define-se a noção de categoria protomodular e mostra-se que esta equivale ao lema cindido dos cinco numa categoria pontuada.

Se \C\ é uma categoria finitamente completa, \PtC\ denota a categoria cujos objectos são os epimorfismos cindidos e cujos morfismos são pares de morfismos (no diagrama $(\gamma_1,\gamma_2)$) que tornam o diagrama comutativo:

\[
\begin{tikzpicture}
\matrix (m) [matrix of math nodes, row sep=1cm, column sep=1cm] 
{X & X' \\ Y & Y'\\};
\path[ar]
(m-1-1) edge node[auto] {$\gamma_1$} (m-1-2)
([xshift=0.5mm] m-1-1.south) edge node[auto] {$f$} ([xshift=0.5mm]m-2-1.north)
([xshift=-0.5mm]m-2-1.north) edge node[auto] {$s$} ([xshift=-0.5mm]m-1-1.south)
(m-2-1) edge node[auto,swap] {$\gamma_2$} (m-2-2)
([xshift=0.5mm] m-1-2.south) edge node[auto] {$f'$} ([xshift=0.5mm] m-2-2.north)
([xshift=-0.5mm] m-2-2.north) edge node[auto] {$s'$} ([xshift=-0.5mm] m-1-2.south);
\end{tikzpicture}
\]
Para cada objecto $X,$ $\Pt{X}$ é definida como sendo a categoria cujos objectos os de \PtC\ com codomínio $X,$ e cujos morfismos são os morfismos $(\gamma_1,\gamma_2)$ com $\gamma_2=1_X.$ Cada morfismo $v:X\to Y$ em \C\ induz um functor de mudança de base:
\[v^*: \Pt{Y} \to \Pt{X},\]
através de produtos fibrados como mostra o diagrama seguinte:
\begin{equation}\begin{tikzpicture}[scale=1.4]
\node (p2) at (xyz cs:y=3,x=0,z=2.5) {$X\times_Y B $};
\node (p1) at (xyz cs:y=2,x=0,z=-2) {$X\times_Y B' $};
\node (b1) at (xyz cs:y=2,x=4,z=-2) {$ B'$};
\node (b2) at (xyz cs:y=3,x=4,z=2.5) {$ B$};
\node (x) at (xyz cs:y=0,x=0,z=0) {$ X$};
\node (y) at (xyz cs:y=0,x=4,z=0) {$ Y.$};

\path[ar] 
(p1) edge (b1)
(p2)edge (b2)
(p2)edge node[auto]{$v^*\gamma$}(p1)
(p1)edge node[auto]{$v^*g'$}(x)
(p2)edge node[auto,swap]{$v^*g$}(x)
(x)edge node[auto,swap]{$v$}(y)
(b1)edge node[auto]{$g'$}(y)
(b2)edge node[auto,swap]{$g$}(y)
(b2)edge node[auto]{$\gamma$}(b1);
\end{tikzpicture}
\label{pm}
\end{equation}

No diagrama $v^*\gamma$ é o morfismo único garantido pela propriedade universal e os morfismos: \begin{align*} v^*g:\ &X \times_Y B \to X\\ v^*g':\ & X\times_Y B' \to X \end{align*} são imagens dos epimorfismos cindidos \(g:B \to Y\), $g':B'\to Y$ que são igualmente epimorfismos cindidos.

\begin{defc}
Uma categoria \C\ finitamente completa diz-se protomodular se no diagrama:
\label{def_protomodular}
\[
\begin{tikzpicture}
\matrix (m) [matrix of math nodes, column sep = 8mm, row sep=8mm] {A & & B & & C\\ &\fbox{1}&&\fbox{2}& \\A'&&B'&&C'\\};
\draw[ar]
(m-1-1) edge node[auto]{$a$} (m-1-3)
(m-1-3) edge node[auto]{$b$}(m-1-5)
(m-3-1) edge node[auto,swap]{$a'$}(m-3-3)
(m-1-1) edge node[auto,swap]{$f$} (m-3-1)
(m-1-5) edge node[auto]{$h$} (m-3-5)
(m-3-3) edge node[auto,swap]{$b'$} (m-3-5)
([xshift=-0.5mm]m-1-3.south) edge node[auto,swap]{$g$} ([xshift=-0.5mm]m-3-3.north)
([xshift=0.5mm]m-3-3.north)  -- node[auto,swap]{$s$}([xshift=0.5mm]m-1-3.south);
\end{tikzpicture}
\]
onde $g$ é um epimorfismo cindido. Se $\fbox{1}$ e $\fbox{1}\fbox{2}$ forem produtos fibrados, então $\fbox{2}$ também o será.
\end{defc}

Nesta secção mostra-se que ser protomodular equivale a satisfazer o lema cindido dos cinco («split short five lemma»).
\begin{pro}
    Numa categoria finitamente completa \C, as afirmações seguintes equivalem-se:
    \begin{enumerate}
        \item \C\ ser protomodular
        \item Para todo o morfismo $v:X\to Y$ em \C, o functor $v^*:\Pt{Y}\to \Pt{X}$ ser conservativo (reflectir isomorfismos).
    \end{enumerate}
\end{pro}

\begin{proof}
    Sejam \C\ uma categoria protomodular, $v:X\to Y$ um morfismo em \C\ e
\[
\begin{tikzpicture}[scale=2]
\matrix(m) [matrix of math nodes, column sep=1cm, row sep=1cm] {X\times_Y B & B & B' \\ X & Y & Y\\};
\path[ar] (m-1-1) edge (m-1-2) (m-1-2) edge node[auto] {$\gamma$} (m-1-3) (m-1-1) edge (m-2-1) (m-2-1) edge node[auto,swap] {$v$} (m-2-2) (m-1-2) edge node[auto,swap] {$g$} (m-2-2) (m-1-3) edge node[auto] {$g'$} (m-2-3);
\path (m-2-2) edge[double, double distance=1mm] (m-2-3);
\end{tikzpicture}
\]
um diagrama em \C. Se $v^*(\gamma)$ (o morfismo único entre $X\times_Y B \to X\times_Y B'$ dado pela propriedade de produto fibrado) for um isomorfismo, $\fbox{1}\fbox{2}$ e $\fbox{1}$ são produtos fibrados dando a conclusão pretendida.

Para provar a outra implicação considere-se o diagrama a seguir:

\[
\begin{tikzpicture}
\matrix (m) [matrix of math nodes, column sep=1.6cm, row sep=1.6cm]
{A & B \\
& & B'\times_{C'} C & C \\
A'& & B' &C'. \\
};
\path[ar]
(m-1-1) edge node[auto]{$a$} (m-1-2)
(m-1-1) edge node[auto,swap]{$f$} (m-3-1)
(m-1-2) edge node[auto]{$b$} (m-2-4)
(m-1-2) edge node[auto,swap]{$g$} (m-3-3)
(m-3-3) edge node[auto,swap]{$b'$}  (m-3-4)
(m-3-1) edge node[auto,swap]{$a'$} (m-3-3)
(m-2-4) edge node[auto]{$h$} (m-3-4)
(m-2-3) edge (m-3-3)
(m-2-3) edge node[auto]{$t$} (m-2-4);
\path[ar,dashed] (m-1-2) edge node[auto,yshift=-1.3mm] {$\lambda$}(m-2-3);
\end{tikzpicture}
\]

Tal como no diagrama acima, acrescentamos o produto fibrado $B'\times_{C'} C$ ao diagrama de Definição \ref{def_protomodular}. Mostre-se que $\lambda$ é um isomorfismo.

Do quadrado comutativo, \[ \begin{tikzpicture}
\matrix (m) [matrix of math nodes, row sep=1.5cm, column sep=1.5cm, text height=1.5ex, text depth=0.25ex]{B' & C \\ B' & C',\\};
\path[ar] (m-1-1) edge node[auto]{$t\circ\lambda\circ s$} (m-1-2)
(m-1-1) edge node[auto,swap]{$1$} (m-2-1)
(m-1-2) edge node[auto]{$h$} (m-2-2)
(m-2-1) edge node[auto,swap]{$b'$}(m-2-2);
\end{tikzpicture} \]
resulta que $B'\times_{C'} C \to B'$ é cindido que permite aplicar o facto de $(a')^*$ ser conservativo.

De \fbox{1}\fbox{2} e \fbox{1} serem produtos fibrados deduz-se $A \cong A' \times_{B'} (B\times_{C'} C),$ o que implica o pretendido.
\end{proof}

Para uma categoria pontuada (uma categoria que tenha um objecto zero), par\-ti\-cu\-lariza-se que protomodularidade implica que, para todo o objecto $X$ em \C, $0^*_X$ é conservativo. Esta condição implica protomodularidade porque, para $v:X\to Y,$ $0_X = v \circ 0_Y,$ logo, $0_X^*= 0_X^* \circ v^*$ e $v^*$ é conservativo, quando $0^*_X$ e $0^*_Y$ forem. Desta caracterização sai de modo natural o lema cindido dos cinco. Primeiramente deve ser definida a noção de núcleo numa categoria pontuada.

\begin{defc}
Numa categoria pontuada \C\ o núcleo de um morfismo $f:X\to Y$ (denotado por (K[f],$\ker f$)) é o seu produto fibrado com o morfismo zero.
\end{defc}

Neste caso, escrevendo $\ker_X:\Pt{X} \to \Pt{0}$ em vez de $0_X^*$ obtemos uma particularidade do diagrama \ref{pm}:
\label{letsusethis}


\[\begin{tikzpicture}[scale=1.65]
\node (p2) at (xyz cs:y=3,x=0,z=2.5) {$K[f] $};
\node (p1) at (xyz cs:y=2,x=0,z=-2) {$K[f'] $};
\node (b1) at (xyz cs:y=2,x=4,z=-2) {$ X'$};
\node (b2) at (xyz cs:y=3,x=4,z=2.5) {$ X$};
\node (x) at (xyz cs:y=0,x=0,z=0) {$ 0$};
\node (y) at (xyz cs:y=0,x=4,z=0) {$ Y$};

\path[ar] 
(p1) edge (b1)
(p2)edge (b2)
(p2)edge[dashed] node[auto]{$\ker_X\gamma$}(p1)
(p1)edge node[auto]{$!$}(x)
(p2)edge node[auto,swap]{$!$}(x)
(x)edge node[auto,swap]{$0_Y$}(y)
(b1)edge node[auto]{$f'$}(y)
(b2)edge node[auto,swap]{$f$}(y)
(b2)edge node[auto]{$\gamma$}(b1);
\end{tikzpicture}
\]

Se numa categoria todo o functor $\ker_X$ reflectir isomorfismos, então deduz-se que o teorema seguinte é válido.
\begin{samepage}
\begin{tma}
Uma categoria pontuada \C\ é protomodular se e só se em cada diagrama da forma:

\[
\begin{tikzpicture}[scale=1.65]
\matrix (m) [matrix of math nodes, column sep=1.2cm,row sep=1.2cm, text height=1.5ex, text depth=0.25ex]
{K[f] & X & Y\\
K[f'] & X' & Y'\\};
\path[ar] 
(m-1-1) edge node[auto]{$\ker f$}(m-1-2)
(m-1-1) edge node[auto,swap]{$a$} (m-2-1)
(m-1-2) edge node[auto]{$f$} (m-1-3)
(m-1-2) edge node[auto,swap]{$b$} (m-2-2)
(m-1-3) edge node[auto]{$c$} (m-2-3)
(m-2-2) edge node[auto,swap]{$f'$} (m-2-3)
(m-2-1) edge node[auto,swap]{$\ker f'$} (m-2-2);
\end{tikzpicture}
\]
com $f$ e $f'$ epimorfismos cindidos, se $a$ e $c$ forem isomorfismos, então $b$ é um isomorfismo.
\end{tma}
\end{samepage}

Nem toda a variedade de álgebras (no sentido de álgebra universal) é protomodular. No entanto, há uma caracterização simples em termos das suas operações (Teorema \ref{this},\cite{Bourn4}). Em particular, as álgebras com uma operação de grupo são protomodulares. 

\section{Categorias homológicas}
\label{cat_hom}
O propósito desta secção é demostrar que uma categoria pontuada, regular e protomodular é um ambiente propício para o estudo de homologia. Designam-se tais categorias homológicas segundo \cite{Bourn}. 

Começa-se por provar o lema dos cinco para uma categoria homológica:

\begin{tma}{Lema dos Cinco\\}
Sendo \C\ uma categoria homológica, e $f$ um epimorfismo regular:
\[
\begin{tikzpicture}[scale=1.65]
\matrix (m) [matrix of math nodes, column sep=1.2cm,row sep=1.2cm, text height=1.5ex, text depth=0.25ex]
{K[f] & X & Y\\
K[f'] & X' & Y'\\};
\path[ar] 
(m-1-1) edge node[auto,swap]{$a$} (m-2-1)
(m-1-2) edge node[auto,swap]{$b$} (m-2-2)
(m-1-3) edge node[auto]{$c$} (m-2-3)
(m-2-2) edge node[auto,swap]{$f'$} (m-2-3);
\path[rar]
(m-1-2) edge node[auto]{$f$} (m-1-3);

\path[nar]
(m-1-1) edge node[auto]{$\ker f$}(m-1-2)
(m-2-1) edge node[auto,swap]{$\ker f'$} (m-2-2);
\end{tikzpicture}
\]
Se $a$ e $c$ forem isomorfismos, então $b$ é um isomorfismo.
\end{tma}

\begin{pro}\label{propo}
Seja \C\ uma categoria homológica. No diagrama seguinte:
\[
\begin{tikzpicture}
\matrix (m) [matrix of math nodes, column sep=.6cm,row sep=.6cm, text height=1.5ex, text depth=0.25ex]
{\cdot&&\cdot&&\cdot\\
&\fbox{\rm{1}}&&\fbox{\rm 2}&\\
\cdot&&\cdot&&\cdot\\};
\path[ar]
(m-1-1) edge (m-1-3)
(m-1-3) edge (m-1-5)
(m-3-1) edge (m-3-3)
(m-3-3) edge (m-3-5)
(m-1-1) edge node[auto,swap]{$f$} (m-3-1)
(m-1-3) edge node[auto,swap]{$g$} (m-3-3)
(m-1-5) edge node[auto,swap]{$h$} (m-3-5);
\end{tikzpicture}
\]
onde $g$ é regular, o facto de $\fbox{\rm1}$ e $\fbox{\rm1}\fbox{\rm2}$ serem produtos fibrados implica que $\fbox{\rm2}$ seja produto fibrado.
\end{pro}

\begin{proof}
Prova-se, de modo directo, que no diagrama abaixo se o quadrado do lado direito for um produto fibrado,
\[
\begin{tikzpicture}
\matrix (m) [matrix of math nodes, column sep=1.4cm, row sep=1.4cm, text height=1.5ex, text depth=0.25ex]{R[f] & \cdot & \cdot\\ R[g]&\cdot&\cdot\\};
\path[ar]
(m-1-1) edge (m-2-1) 
(m-1-2) edge (m-2-2)
(m-1-3) edge (m-2-3)
(m-1-2) edge node[auto]{$f$} (m-1-3)
(m-2-2) edge node[auto,swap]{$g$} (m-2-3)
([yshift=.5mm]m-1-1.east) edge  ([yshift=.5mm]m-1-2.west)
([yshift=-.5mm]m-1-1.east) edge  ([yshift=-.5mm]m-1-2.west)
([yshift=.5mm]m-2-1.east) edge ([yshift=.5mm]m-2-2.west)
([yshift=-.5mm]m-2-1.east) edge ([yshift=-.5mm]m-2-2.west);
\end{tikzpicture}
\]
então os quadrados do lado esquerdo também o serão.

Como consequência deste facto no diagrama estendido abaixo
\[
\begin{tikzpicture}
\matrix (m) [matrix of math nodes, column sep=.4cm, row sep=.4cm, text height=1.5ex, text depth=0.25ex]
{R[f] & & R[g] && R[h]\\ 
& \node[draw]{3}; && \node[draw]{4}; & \\
\cdot&&\cdot&&\cdot\\
&\node[draw]{1};&&\node[draw]{2};&\\
\cdot&&\cdot&&\cdot\\};
\path[ar]
(m-1-1) edge (m-1-3)
(m-1-3) edge (m-1-5)
(m-3-1) edge (m-3-3)
(m-3-3) edge (m-3-5)
([xshift=.5mm]m-1-1.south) edge  ([xshift=.5mm]m-3-1.north)
([xshift=-.5mm]m-1-1.south) edge  ([xshift=-.5mm]m-3-1.north)
([xshift=.5mm]m-1-3.south) edge ([xshift=.5mm]m-3-3.north)
([xshift=.5mm]m-1-5.south) edge ([xshift=.5mm]m-3-5.north)
([xshift=-.5mm]m-1-3.south) edge ([xshift=-.5mm]m-3-3.north)
([xshift=-.5mm]m-1-5.south) edge ([xshift=-.5mm]m-3-5.north)

(m-3-1) edge (m-3-3)
(m-3-3) edge (m-3-5)
(m-5-1) edge (m-5-3)
(m-5-3) edge (m-5-5)
(m-3-1) edge node[auto,swap]{$f$} (m-5-1)
(m-3-3) edge node[auto,swap]{$g$} (m-5-3)
(m-3-5) edge node[auto,swap]{$h$} (m-5-5);
\end{tikzpicture}
\]
os quatro quadrados $\fbox{3}$ e $\fbox{3}\fbox{4}$ são produtos fibrados (pois $\fbox{1}$ e $\fbox{1}\fbox{2}$ o são), logo, como a diagonal é uma inversa à direita da projecção de uma relação de equivalência, aplica-se a definição de protomodular para concluir que $\fbox{4}$ é um produto fibrado. A aplicação do teorema de Barr-Kock (Teorema \ref{bk}) mostra que $\fbox{2}$ é um produto fibrado.
\end{proof}

Termina-se com a demonstração do Lema dos Cinco e outros lemas de homologia numa categoria homológica.
\begin{proof}{(do Lema dos Cinco)}\\
No diagrama seguinte:
\[
\begin{tikzpicture}
\matrix (m) [matrix of math nodes, column sep=1.2cm, row sep=1.2cm, text height=1.5ex, text depth=0.25ex]
{K[f] & X & Y\\
K[f'] & X' & Y'\\
};
\path[ar]
(m-2-2) edge node[auto,swap]{$f'$} (m-2-3)
(m-1-2) edge node[auto]{$b$} (m-2-2);
\path[nar]
(m-1-1) edge node[auto]{$\ker f$} (m-1-2)
(m-2-1) edge node[auto,swap]{$\ker f'$} (m-2-2);
\path[rar]
(m-1-2) edge node[auto]{$f$} (m-1-3);
\path[ar]
(m-1-1) edge node[auto]{$a$} (m-2-1)
(m-1-3) edge node[auto]{$c$} (m-2-3);
\end{tikzpicture}
\]
com $f$ um epimorfismo regular, e $a$ e $c$ isomorfismos, é necessário mostrar que $b$ é um isomorfismo. Em primeiro lugar estendamos o diagrama:

\[
\begin{tikzpicture}
\matrix (m) [matrix of math nodes, column sep=.6cm, row sep=.6cm, text height=1.5ex, text depth=0.25ex] {&&K[f]&&0\\
&&&\node[draw]{1};&\\
K[f]&&X&&Y\\
&\node[draw]{3};&&\node[draw]{2};&\\
K[f']&&X'&&Y'.\\};

\path[ar]
(m-5-3) edge node[auto,swap]{$f'$} (m-5-5)
(m-3-3) edge node[auto,swap]{$b$} (m-5-3);
\path[nar]
(m-3-1) edge node[auto]{$\ker f$} (m-3-3)
(m-5-1) edge node[auto,swap]{$\ker f'$} (m-5-3);
\path[rar]
(m-3-3) edge node[auto]{$f$} (m-3-5);
\path[ar]
(m-3-1) edge node[auto,swap]{$a$} (m-5-1)
(m-3-5) edge node[auto]{$c$} (m-5-5);

\path[ar]
(m-1-3) edge (m-1-5)
(m-1-5) edge (m-3-5);
\path[nar]
(m-1-3) edge node[auto,swap] {$\ker f$} (m-3-3);
\end{tikzpicture}
\]
Ora, $\fbox{1}$ é um produto fibrado, e como $a$ e $c$ são isomorfismos, {\tiny ${\renewcommand{\arraystretch}{.04}
\renewcommand{\tabcolsep}{0.2cm}
\begin{array}{c}\fbox{1} \\  \fbox{2} \end{array}}$}\normalsize também é um produto fibrado. A proposição \ref{propo} diz que $\fbox{2}$ é um produto fibrado, logo $b$ é um isomorfismo.
\end{proof}

Como evidência da utilidade da noção de categoria homológica, tem-se o terceiro teorema de isomorfismo que é válido para categorias homológicas.

\begin{defc}
Numa categoria homológica, um subobjecto, $X\subseteq Y,$ diz-se próprio se for um núcleo que é designado $X\unlhd Y.$
\end{defc}

\begin{tma}
Seja \C\ uma categoria homológica, e sejam $H\unlhd G$ e $K\unlhd G$ subobjectos próprios de $G$, e $H\subseteq K$ um subobjecto de $K.$ Nestas condições vigoram os resultados seguintes: 
\begin{enumerate}
\item $H$ é um subobjecto próprio de K; 
\item $K/H \unlhd G/H $; 
\item $(G/H)/(K/H) \cong G/K$.
\end{enumerate}
\end{tma}

\begin{proof} 
Sejam $h:H\to G$ o morfismo $$H\mathrel{\mathop{\rightarrow}^i} K \mathrel{\mathop{\to}^k} G.$$ Considere-se o diagrama seguinte:
\[
\begin{tikzpicture}[description/.style={fill=white,inner sep=2pt}]
\matrix (m) [matrix of math nodes, row sep=4em,
column sep=3.4em, text height=1.5ex, text depth=0.25ex]
{ H & G & G/H \\
& G/K & \\ };
\path[ar]
(m-1-1) edge node[auto,swap]{$0$} (m-2-2);
\path[nar]
(m-1-1) edge node[auto]{$h$} (m-1-2);
\path[rar]
(m-1-2) edge node[auto,swap]{$q$} (m-1-3)
(m-1-2) edge node[auto]{$p$} (m-2-2)
(m-1-3) edge[dashed] node[auto]{$\phi$} (m-2-2);
\end{tikzpicture}
\]

Como $H\subseteq K$, $p\circ h$ é o morfismo zero e portanto existe um único morfismo $\phi: G/H \to G/K$ é o único morfismo que torna o diagrama comutativo. 

Ora como, $$\phi\circ q \circ k \circ i = \phi\circ q \circ h = 0$$ e $i:H\to K$ é um monomorfismo, obtém-se uma factorização através dos núcleos:
\[
\begin{tikzpicture}
\matrix(a) [matrix of math nodes,row sep=4em,
column sep=4em, nodes in empty cells] { 
H & K & \ker \phi\\
H & G & G/H \\
0 & G/K & G/K\\};
\path[ar,dashed] (a-1-2) edge node[auto]{$j$} (a-1-3);
\path (a-1-1) edge[double] (a-2-1);
\path (a-3-2) edge[double] (a-3-3);

\path[nar] (a-1-1) edge node[auto]{$i$} (a-1-2);
\path[nar] (a-2-1) edge node[auto]{$h$} (a-2-2);
\path[nar] (a-3-1) edge (a-3-2);
\path[nar] (a-1-2) edge node[auto]{$k$} (a-2-2);
\path[nar] (a-1-3) edge  (a-2-3);

\path[rar] (a-2-2) edge node[auto]{$q$} (a-2-3);
\path[rar] (a-2-1) edge (a-3-1);
\path[rar] (a-2-2) edge node[auto]{$p$} (a-3-2);
\path[rar] (a-2-3) edge node[auto]{$\phi$} (a-3-3);
\end{tikzpicture}
\]

Como as colunas são exactas e as últimas linhas também o são, um caso especial do lema dos nove, anunciado e demonstrado a seguir, justifica o resultado.
\end{proof}

\begin{lem}{Caso particular do Lema dos Nove}

Num diagrama comutativo $3\times 3,$ se as últimas duas linhas e todas as colunas são exactas então a primeira linha também é exacta.

\[
\begin{tikzpicture}
\matrix (a) [matrix of math nodes,row sep=5em,
column sep=5em, nodes in empty cells]
{ A&B&C \\ A'&B'&C' \\ A''&B''&C''\\ };
\path[ar] (a-1-1) edge node[auto]{$f$} (a-1-2);
\path[ar] (a-1-2) edge node[auto]{$g$}(a-1-3);
\path[nar] (a-2-1) edge node[auto]{$f'$} (a-2-2);
\path[rar] (a-2-2) edge node[auto]{$g'$}(a-2-3);
\path[nar] (a-3-1) edge node[auto,swap]{$f''$} (a-3-2);
\path[rar] (a-3-2) edge node[auto,swap]{$g''$}(a-3-3);

\path[nar] (a-1-1) edge node[auto,swap]{$a$} (a-2-1);
\path[rar] (a-2-1) edge node[auto,swap]{$a'$}(a-3-1);

\path[nar] (a-1-2) edge node[auto]{$b$} (a-2-2);
\path[rar] (a-2-2) edge node[auto]{$b'$}(a-3-2);

\path[nar] (a-1-3) edge node[auto]{$c$} (a-2-3);
\path[rar] (a-2-3) edge node[auto]{$c'$}(a-3-3);
\end{tikzpicture}
\]
\end{lem}

\begin{proof}
Como as primeiras colunas são exactas e $f''$ ser um monomorfismo implica que o primeiro quadrado $A\cong A'\times_{B'}B$ é um produto fibrado. Daí vem que $f=\ker g$. Verifica-se ainda que $g$ é um epimorfismo regular, razão pela qual a primeira linha é exacta.
\end{proof}

\section{Categorias semi-abelianas}
A definição de categoria semi-abeliana como proposta em \cite{Janelidze} é uma combinação das propriedades que estudámos até agora. Especificamente, uma categoria pontuada \C\ diz-se semi-abeliana se for cocompleta, exacta e protomodular. Como exemplo imediato têm-se as categorias abelianas. 

\begin{ex}
Toda a categoria abeliana é semi-abeliana.
\end{ex}

\begin{proof}
Uma categoria abeliana é, por definição, uma categoria pontuada, finitamente completa e cocompleta, com núcleos e conúcleos, e onde todo o monomorfismo é um núcleo e todo o epimorfismo é um conúcleo.

Uma categoria abeliana satisfaz o lema dos cinco que é mais forte do que ser protomodular. Extrai-se directamente do facto de ser semi-abeliana que é regular: existe uma factorização de cada morfismo $f=\im f\circ \textrm{coim} f,$ única a menos de isomorfismo. Por fim, a categoria é exacta, pois cada relação $R$ é isomorfa a $R[q]$, a relação efectiva gerada pelo seu co-igualizador $q$.
\end{proof}

As categorias abelianas têm a propriedade de que as suas opostas também são abelianas. No caso das semi-abelianas já não acontece, porém tem-se:

\begin{pro}
Uma categoria $\C$ semi-abeliana e cuja categoria oposta $\C^\mathrm{op}$ seja semi-abeliana é abeliana.
\end{pro}

\begin{proof}
Remetemos o leitor para \cite{Janelidze}.
\end{proof}


Por fim, consideramos outras aplicações de categorias semi-abelianas. Na tese \cite{vanLinden}, é notável que as categorias semi-abelianas formam um ambiente propício para o estudo de homologia e homotopia. O artigo \cite{Everaert} revela que se pode estudar a teoria dos comutadores nas categorias semi-abelianas.

\chapter{Álgebras topológicas}
Este capítulo trata os espaços topológicos munidos de uma estrutura algébrica ge\-ne\-ra\-li\-zan\-do propriedades clássicas de grupos topológicos para álgebras topológicas. Com esse fim, recordam-se brevemente os resultados conhecidos dos grupos topológicos nos contextos da teoria das categorias e da álgebra universal. De seguida, usa-se um resultado de Bourn e Janelidze \cite{Bourn4} para provar as generalizações dos resultados de $\GrpTop$ como em \cite{Borceux3}. 

\section{Grupos topológicos, $\GrpTop$}

Grupos topológicos são simultaneamente espaços topológicos e grupos com uma condição de compatibilidade: as operações de grupo são contínuas. No entanto, podem definir-se da maneira seguinte.

\begin{defc}
Um grupo topológico é um par $(X,\frak{T})$ onde $X$ é um grupo cujo conjunto subjacente é munido de uma topologia tal que $\phi:X\times X\to X;\ (x,y)\mapsto xy^{-1}$ é uma aplicação contínua (em $X\times X$ considera-se a topologia produto). 

Um morfismo de grupos topológicos é um homomorfismo de grupos que seja contínuo como aplicação entre os seus espaços.
\end{defc}


As topologias de grupos topológicos são muito próprias, como mostra o lema seguinte.

\begin{lem}
Todo o grupo topológico é um espaço homogéneo.
\end{lem}

\begin{proof}
Seja $X$ um grupo topológico, e sejam  $a,b\in X.$ Mostre-se que a aplicação $X\to X;\ t\mapsto b\inv{a} t$, que a $a$ atribui $b$, é um homeomorfismo. É uma aplicação contínua, pois a multiplicação é uma operação contínua, e $X \to X;\ t \mapsto ab^{-1}t$, a sua inversa, é igualmente contínua.
\end{proof}

\begin{cor}
Seja $X$ um grupo topológico. Se uma propriedade topológica $\mathfrak{P}$ é válida numa vizinhança aberta do elemento neutro de  $X$, então é válida em $X$.
\end{cor}

\begin{proof}
Sabe-se que os homeomorfismos entre espaços topológicos preservam as suas propriedades topológicas. Seja $f:X\to X$ um homeomorfismo de grupos topológicos que a $0$ atribui $x\in X$. Se $\mathfrak{P}$ for válida numa vizinhança aberta $U$ de $0,$ então, $f(U)$ é uma vizinhança aberta de $x$ onde $\mathfrak{P}$ é ainda válida.
\end{proof}

\begin{pro}
Os grupos topológicos e homomorfismos de grupo contínuos entre grupos topológicos constituem uma categoria finitamente completa que se designa $\GrpTop.$
\end{pro}

\begin{proof}

Construam-se os produtos binários e igualizadores explicitamente. 

Se $X$ e $Y$ são objectos de $\GrpTop,$ o produto de $X$ e $Y$ em $\Grp$ será o conjunto $\{(x,y) : x\in X, y\in Y\}$ munido do produto $(x,y)\cdot (x',y') = (x x',y y'),$ com as projecções canónicas. 

\[\begin{tikzpicture}
\matrix (m) [matrix of math nodes, column sep=1cm] {X & X\times Y & Y\\};
\path[ar]
(m-1-2) edge node[auto,swap]{$p_X$}(m-1-1)
(m-1-2) edge node[auto]{$p_Y$}(m-1-3);
\end{tikzpicture} \]

Verifica-se que a topologia produto $X\times Y$ torna-o num grupo topológico, pois o morfismo $$((x_1,x_2),(y_1,y_2)) \mapsto (x_1x_2^{-1}, y_1y_2^{-1})$$ é contínuo em cada componente. As projecções são contínuas, pois verifica-se que $p_X^{-1}(U) = U\times Y.$

O que é o igualizador de $f,g:X\to Y$? O igualizador \[
\begin{tikzpicture}
\matrix (m) [matrix of math nodes, column sep=1.5cm,text height=1.5ex, text depth=0.25ex] {I & X & Y,\\};
\path[ar]
(m-1-1) edge node[auto]{$i$}(m-1-2)
([yshift=.5mm]m-1-2.east) edge node[auto]{$f$}([yshift=.5mm]m-1-3.west)
([yshift=-.5mm]m-1-2.east) edge node[auto,swap]{$g$}([yshift=-.5mm]m-1-3.west);
\end{tikzpicture}
\]
será o subgrupo topológico de $X$ cujos elementos são os elementos $x$ tal que $f(x) = g(x).$ Averigua-se que é um subgrupo e a inclusão é um homomorfismo contínuo. 

\end{proof}

Para facilitar o uso de produtos fibrados em $\GrpTop,$ escreve-se o produto fibrado de dois morfimos em $\GrpTop$ de uma forma explícita. Sejam $X,$ $Y$ e $Z$ grupos topológicos com morfismos $X\mathop{\to}\limits^f Z \mathop{\leftarrow}\limits^g Y.$ O produto fibrado \[
\begin{tikzpicture}
\matrix (m) [matrix of math nodes, column sep=1.5cm, row sep=1.5cm] {X\times_Z Y & Y\\ X & Z.\\};
\path[ar]
(m-1-1) edge node[auto] {$p_2$} (m-1-2)
(m-1-1) edge node[auto,swap] {$p_1$} (m-2-1)
(m-2-1) edge node[auto,swap] {$f$} (m-2-2)
(m-1-2) edge node[auto] {$g$} (m-2-2);
\end{tikzpicture}
 \] tem como conjunto subjacente do produto fibrado é $$|X\times_Z Y| = \{(s,t)\in |X\times Y| : f(s)=g(t)\},$$ e a topologia é induzida pela topologia produto.

\begin{pro}
$\GrpTop$ é finitamente cocompleta.
\end{pro}

\begin{proof}
Um coproduto de grupos topológicos é o produto livre dos grupos munido da topologia gerada pela união disjunta das topologias. Para mais pormenores, remetemos o leitor para \cite{Borceux}.

O co-igualizador $Q$ \[
\begin{tikzpicture}
\matrix (m) [matrix of math nodes, column sep=1.5cm,text height=1.5ex, text depth=0.25ex] { X & Y & Q.\\};
\path[ar]
(m-1-2) edge node[auto]{$q$}(m-1-3)
([yshift=.5mm]m-1-1.east) edge node[auto]{$f$}([yshift=.5mm]m-1-2.west)
([yshift=-.5mm]m-1-1.east) edge node[auto,swap]{$g$}([yshift=-.5mm]m-1-2.west);
\end{tikzpicture}
\]
é dado pelo quociente $Y/_\sim$ onde $\sim$ é a menor relação de equivalência compatível com as operações $(*,^{-1},e)$ que satisfaz $f(x)\sim g(x)$ para $x\in X,$ e é munido da topologia quociente. Esta estrutura é naturalmente um grupo topológico.

\end{proof}

Para posteriormente compararmos com as técnicas mais sofisticadas (que se empregam para estudar álgebras topológicas), incluem-se aqui alguns resultados fundamentais da teoria dos grupos topológicos. Como o foco são os métodos, as de\-mons\-tra\-ções são completas. Esta matéria pode ser encontrada em várias fontes, veja-se, por exemplo, \cite{Higgins,Bourbaki}.

O objecto com um único elemento é o objecto zero, portanto, $\GrpTop$ é pontuada. Para economizar espaço, omitimos provar que $\GrpTop$ é regular e protomodular. Segue das demonstrações dos Teoremas \ref{l1} e \ref{l2}.

\begin{pro}
A categoria dos grupos topológicos é homológica.
\end{pro}



\begin{cex}
O primeiro teorema de isomorfismo para grupos não é válido em $\GrpTop.$ Isso é uma consequência do facto de uma bijecção contínua nem sempre ser um isomorfismo em $\GrpTop$ (Os isomorfismos em $\GrpTop$ são simultaneamente homeomorfismos de espaços topológicos e homomorfismos de grupos):
Considere-se o morfismo de inclusão dos conjuntos, $$f:G=(S^1,\mathcal{P}(S^1)) \rightarrow H=(S^1, \text{top. ind. por } \mathbb{R/Z}).$$ 
O uso do teorema de isomorfismo canónico em $f$ implica $G\cong H$ o que é absurdo, sendo as topologias diferentes.
\end{cex}

\begin{tma}{\rm (1º Teorema de Isomorfismo em $\GrpTop$)\\}
Sejam $f:X\to Y$ um morfismo de grupos topológicos, $(K[f],\ker f)$ o núcleo de $f$ e $N\lhd X.$ Então, $N\subset K[f]$ se e somente se existir um morfismo único $\tilde f: X/N \to Y$ que torne o diagrama seguinte comutativo: \[
\begin{tikzpicture}
\matrix(m) [matrix of math nodes,column sep=1.7cm,row sep=1.7cm,text height=1.5ex, text depth=0.25ex]{X & X/N \\ Y& .\\};
\path[rar]
(m-1-1) edge node[auto]{$\pi$} (m-1-2);
\path[ar]
(m-1-1) edge node[auto,swap]{$f$} (m-2-1);
\path[nar,dashed]
(m-1-2) edge node[auto]{$\tilde f$} (m-2-1);
\end{tikzpicture}
\]
\end{tma}

\begin{proof}
A demonstração é uma cópia traduzida da demonstração habitual do teorema em $\Grp$.
\end{proof}

\begin{cor}{\rm (1º Teorema de Isomorfismo em $\GrpTop$)\\}
Nas condições do teorema anterior, $X/K[f] \cong \im f$ se e somente se $f$ for uma aplicação aberta.
\end{cor}

\begin{proof}
$\tilde{f}:X\to \im f$ é bijectiva e contínua, mas $\tilde{f}^{-1}$ é contínua se e só se $f$ for aberta ($\tilde{f}(U+K[f]) = f(U)$).
\end{proof}

\begin{tma}
Se um grupo topológico tiver um sistema de vizinhanças numerável da identidade, então é metrizável.
\end{tma}

\begin{proof}
Um espaço topológico que tenha um sistema de vizinhanças numerável em cada ponto (isto é, que satisfaça o primeiro axioma de numerabilidade) é metrizável. Deste modo, basta verificar no elemento neutro.
\end{proof}

\begin{pro}
Um subgrupo $K\leq X$ é normal em $\GrpTop$ se e só se a inclusão $i:K\hookrightarrow X$ for um núcleo.
\end{pro}

\begin{proof}
($\Rightarrow$)

Seja $K\lhd X$ um subgrupo normal de $X.$ Mostre-se que a inclusão de $K$ em $X$ é um núcleo. Forma-se o grupo $X/K$ (o qual juntamente com a projecção natural é o conúcleo de $i:K\to X,$ munido da topologia quociente). Seja $f:Y\to X$ um morfismo com $\pi\circ f = 0.$ Existe a única função $t:Y\to K,$ que torna o diagrama comutativo:
\[
\begin{tikzpicture}
\matrix (m) [matrix of math nodes, column sep=1.5cm, row sep= 1.5cm,text height=1.5ex, text depth=0.25ex] {K & X & X/K\\ Y & &.\\};
\path[rar] (m-1-2) edge node[auto]{$\pi$} (m-1-3);
\path[right hook->] (m-1-1) edge node[auto]{$i$} (m-1-2);
\path[ar] (m-2-1) edge node[auto,swap]{$f$} (m-1-2);
\path[ar,dashed](m-2-1) edge node[auto]{$t$} (m-1-1);
\end{tikzpicture}
\]
De modo a averiguar que $t$ é um morfismo em $\GrpTop,$ seja $U\subseteq K$ um aberto. Existe um aberto $V\subset X$ tal que $U=V\cap K,$ portanto $$f^{-1}(V) = (i\circ t)^{-1}(V) = t^{-1}(i^{-1}(V)) = t^{-1}(U),$$ do qual sai que $t$ é contínuo.

($\Leftarrow$)

Se $i:K\to X$ é o núcleo de $q:X\to Q$ obtém-se $q(x^{-1} k x) = q(x^{-1})q(k)q(x) = q(x)^{-1}q(x) = e_Q,$ e daí $x^{-1} k x$ pertence ao núcleo de $i$, portanto $K \lhd X.$
\end{proof}

\begin{tma}
Seja $H\leq X$ um subgrupo de $X.$ Se $H$ for aberto, então é também fechado.
\end{tma}

\begin{proof}
Se $H$ é aberto, então $xH$ é aberto também para $x\in X,$ porém, $$X\setminus H = \bigcup_{x\not\in H} xH$$ que é uma reunião de abertos, é aberto. Logo, $H$ é fechado.
\end{proof}

\begin{tma}
Seja $H\leq X$ um subgrupo topológico, então $\overline{H}$ é também um subgrupo topológico. Além disso, se $H$ for normal então $\overline{H}$ também será.
\end{tma}

\begin{proof}
Directa. Prova-se que as operações se estendem de forma natural para operações contínuas em $\overline{H}.$
\end{proof}

\section{Grupos topológicos como uma variedade de álgebra universal}
Estruturas algébricas são conjuntos munidos de operações que satisfazem certas condições tal como associatividade e comutividade. A álgebra universal tem como objectivo formalizar essas estruturas num estudo abrangente, unificando todos os ramos de álgebra moderna. É útil no contexto de álgebra topológica porque todas as álgebras topológicas podem ser modeladas como álgebras universais internas em $\Top,$ $\mathbf{Haus}$, etc.


\begin{defc}
Uma {\it estrutura algébrica} (ou {\it álgebra}) é um par ordenado \(\underline{A}=(A,F),\) onde $A$, o {\it universo} de $\underline{A}$, é um conjunto e $$F=\{\omega^{A}:A^{n_\omega}\to A \mid \omega \in \Omega \}$$ consiste nas operações finitas de $\underline{A}$ indexadas pelo conjunto $\Omega.$ A {\it aridade} de uma operação $\omega^{B}$ é o número natural $n_\omega.$ A {\it assinatura} de $\underline{A}$ é a função $\tau:\Omega \to \mathbb{N};\ \omega \to n_\omega.$ As operações de aridade $n$ formam um conjunto $\Omega_n.$
\end{defc}

De modo a economizar espaço, daqui em diante omitiremos especificar a função $\tau:\Omega\to \mathbb{N}$ e só irá ser referida uma álgebra de assinatura $\Omega.$

Uma determinada espécie de álgebras, por exemplo, os grupos, têm em comum as suas operações. Há uma correspondência directa entre essa determinada espécie, e as suas operações com as leis que as regem.

As operações de aridade $n$ ou as operações $n$-árias designam-se elementos de $\Omega_n.$ Utiliza-se o termo {\it constante} para as operações de aridade $0$.

\begin{defc}
Um {\it homomorfismo de álgebras de assinatura $\Omega$} é uma função $f:A\to B$ entre os universos das álgebras $\underline{A}$ e $\underline{B}$ que satisfaz, para cada operação $\omega\in \Omega_n,$ $$f\omega^A(a_1,\ldots,a_n)= \omega^B(f(a_1),\ldots,f(a_n)).$$
\end{defc}

Considera-se útil definir as classes de álgebras que satisfaçam uma lista de axiomas. Incluímos os grupos, por exemplo, que satisfazem os três axiomas habituais. Para tal, definimos o conjunto dos termos $T(X).$

\begin{defc}
O conjunto dos {\it termos} nas variáveis, $X$, de assinatura $\Omega$ é definido recursivamente (suponha-se que $X\cap \Omega_0=\emptyset$):
\begin{itemize}
\item Os termos constantes $c\in \Omega_0$ e as variáveis $x\in X$ pertencem a $T(X)$;
\item $f(t_1,\ldots,t_n)\in T(X)$ quando $t_i\in T(X)$ e $f\in \Omega_n$.
\end{itemize}
Define-se a {\it avaliação} dos termos de forma natural.
\end{defc}

Com estes termos definimos os axiomas como sendo {\it equações}: $t_i \approx t_j,$ e diz-se que uma álgebra de assinatura $\Omega$ satisfaz $t_i \approx t_j$ se $t_i(\overline{a}) = t_j(\overline{a}),\ \forall \overline{a}.$ Usa-se o símbolo $\approx$ para realçar a diferença entre equações e a igualdade de elementos do universo.

Neste momento já se consegue definir um dos conceitos centrais de álgebra universal, a variedade.
\begin{defc}
Uma {\it variedade} $\mathcal{V}$ de álgebras universais é a categoria de toda a álgebra sobre uma assinatura $\Omega$ que satisfaz um conjunto de leis $E,$ cujas variáveis são elementos de $X$, bem como os morfismos entre elas. Designa-se uma teoria algébrica pelo terno $\Teoria=(\Omega,X,E).$
\end{defc}

Como exemplo intuitivo, tem-se a variedade dos grupos:
\begin{ex}
Sejam $\Omega_0=\{z\},\ \Omega_1=\{-\}$ e $\Omega_2=\{+\}$ e seja $X=\{a,b,c\}.$ Os axiomas escritos em $a,$ $b$ e $c$ são $$E=\left\{a+(b+c)\approx(a+b)+c,\ a+z\approx a,\ z+a\approx a,\ a+(-a)\approx z,\ (-a)+a\approx z\right\}.$$ $\Teoria=(\Omega,X,E)$ é a teoria dos grupos e a variedade correspondente (que se designa por $\Alg(\Teoria)$) é isomorfa a $\Grp.$
\end{ex}

\begin{pro}
\label{varexacta}
Seja qual for a teoria algébrica, a variedade $\Alg(\Teoria)$ é exacta.
\end{pro}

\begin{proof}
Dado um morfismo $f:\underline{A} \to \underline{B}$ em $\Alg(\Teoria),$ mostre-se que existe a factorização da proposição \ref{factor}. Como $\Conj$ é regular, existe a factorização:

\[\begin{tikzpicture}
\matrix (m) [matrix of math nodes, column sep=1.5cm, row sep=1.5cm]{A & & B \\& Q, &\\};
\path[ar] (m-1-1) edge node[auto]{$|f|$} (m-1-3)
(m-1-1) edge node[auto,swap]{$e$} (m-2-2)
(m-2-2) edge node[auto,swap]{$m$}(m-1-3);
\end{tikzpicture}\]
em $\Conj.$ Mostremos que essa factorização é compatível com as operações de $\Teoria.$

Seja $\tau$ uma operação de aridade $n;$ da equação $|f|\circ \tau = \tau \circ |f|^n$ advém o seguinte diagrama comutativo em $\Conj$:

\[
\label{ref}
\begin{tikzpicture}
\matrix (m) [matrix of math nodes, column sep=1.5cm, row sep=1.5cm,text height=1.5ex, text depth=0.25ex]{A^n & S^n &  B^n \\ A& S & B \\ };
\path[ar] (m-1-1) edge node[auto,swap]{$\tau^A$} (m-2-1)
(m-1-3) edge node[auto]{$\tau^B$} (m-2-3);
\path[ar] (m-1-2) edge (m-2-2);
\path[rar] (m-1-1) edge node[auto]{$e^n$} (m-1-2)
(m-2-1) edge node[auto,swap]{$e$} (m-2-2);
\path[nar] (m-1-2) edge node[auto]{$m^n$} (m-1-3)
(m-2-2) edge node[auto,swap]{$m$} (m-2-3);
\end{tikzpicture}
\]
Para levantar o diagrama acima à variedade é necessário que $e:A\to S$ e $m: S \to B$ definam morfismos em $\Alg(\Teoria)$. Primeiro mostramos que $e^n$ é um epimorfismo regular. Basta considerarmos $e\times e.$

Como a composição de epimorfismos regulares é regular e o produto fibrado de um epimorfismo regular é um epimorfismo regular o produto fibrado que se segue mostra que $e\times e= (1\times e)\circ(e\times 1)$ é um epimorfismo regular.

\[
\begin{tikzpicture}
\matrix (m) [matrix of math nodes, column sep=1.5cm, row sep=1.5cm,text height=1.5ex, text depth=0.25ex]{A \times B & S \times B \\ A & S\\};
\path[ar] (m-1-1) edge node[auto]{$e\times 1$} (m-1-2)
(m-1-1) edge node[auto,swap]{$p_1$} (m-2-1)
(m-2-1) edge node[auto,swap]{$e$} (m-2-2)
(m-1-2) edge node[auto]{$p_2$} (m-2-2);
\end{tikzpicture}
\]

Como os monomorfismos são igualmente estáveis para produtos fibrados, a Pro\-po\-si\-ção \ref{factorim} mostra que a normalidade de $\Conj$ implica que exista uma função única $\tau^S:S^n\to S$ tornando o diagrama \ref{ref} comutativo. Assim, pode levantar-se a factorização de $|f|$ para $\Alg(\Teoria),$ motivo pelo qual é normal.

Seja $r_1,r_2:\underline{R}\to \underline{X}$ uma relação de equivalência em $\Alg(\Teoria).$ Em \Conj\ podemos escrever $q=\coig(|r_1|,|r_2|)$ e portanto, como $\Conj$ é exacta, $|r_1|,|r_2|:R\to X$ é o par núcleo de $q:X\to Q.$ Verifica-se ainda que $|r_1|^n,|r_2|^n: R^n \to X^n$ é o par núcleo de $q^n$ (o produto de um epimorfismo regular é regular). Decorre que $Q$ pode ser munido da estrutura de álgebra sobre $\Teoria$ e, então, $R\rightrightarrows X$ é uma relação de equivalência em $\Alg(\Teoria).$
 
\end{proof}
As álgebras internas, de uma determinada teoria algébrica, são objectos de uma categoria cujas operações são morfismos na categoria. Por exemplo, $\GrpTop$ pode ser visto como a categoria cujos objectos são os de $\Top$ munidos das operações de grupo que são aplicações contínuas. Para uma teoria qualquer, escreve-se $\C^\Teoria$ para denotar a categoria de todas as álgebras sobre $\Teoria$ em $\C.$ Continuamos a chamar-lhes variedades. (Note-se que $\Conj^\Teoria$ é exactamente $\Alg(\Teoria).$)

Utilizar álgebra universal neste contexto supõe certas vantagens: 1) a álgebra universal em si tem muitas ferramentas poderosas ao nosso dispor, 2) é um caminho eficiente para estudar diferentes álgebras topológicas em simultâneo e 3) a ligação entre a álgebra universal e a teoria das categorias já estabelecida permite-nos aproveitar as vantagens da linguagem de categorias.

\section{Álgebras protomodulares}

Anuncie-se um resultado muito conveniente sobre teorias protomodulares (Bourn-Janelidze \cite{Bourn3}) cuja demonstração omitimos.

\begin{tma}\label{this}
Uma variedade $\Alg(\Teoria)$ é protomodular quando para algum número natural $n$, a teoria contém:
\begin{enumerate}
\item $n$ constantes $e_1, \ldots , e_n$;
\item $n$ termos binários $\alpha_1,\ldots, \alpha_n$ com $\alpha_i(x,x)=e_i$;
\item um termo $\theta$ com $\theta(\alpha_1(x,y),\ldots,\alpha_n(x,y),y)=x$.
\end{enumerate}
\end{tma}
Do teorema deduz-se o que se segue.

\begin{cor}
Uma variedade $\Alg(\Teoria)$ é semi-abeliana quando, para algum número natural $n$, a teoria contém:
\begin{enumerate}
\item uma constante $e$;
\item $n$ termos binários $\alpha_1,\ldots, \alpha_n$ com $\alpha_i(x,x)=e$;
\item um termo $\theta$ com $\theta(\alpha_1(x,y),\ldots,\alpha_n(x,y),y)=x$.
\end{enumerate}
\end{cor}

\begin{proof}
Uma variedade é pontuada precisamente quando exista uma cons\-tante única. $\Conj^\Teoria$ é exacta (Proposição \ref{varexacta}).
\end{proof}

\begin{defc}
Diz-se que uma teoria é protomodular (semi-abeliana, abeliana, etc.) quando $\Conj^\Teoria (= \Alg(\Teoria)$) é protomodular (semi-abeliana, abeliana, respec\-ti\-vamente).
\end{defc}

Prove-se de uma forma sucinta que $\GrpTop$ é protomodular e pontuada. A teoria dos grupos é uma teoria protomodular e pontuada. Examinámos anteriormente que os limites em $\GrpTop$ provêm dos limites em $\Grp$, motivo pelo qual $\GrpTop$ é igualmente protomodular.

\begin{ex}
Qualquer teoria que inclua as operações do grupo é semi-abeliana. 
\end{ex}

\begin{proof}
Na caracterização de Bourn-Janelidze basta tomar $n=1$, $\alpha (x,y)=xy^{-1}$ e $\theta (x,y)=xy$.
\end{proof}

\begin{cex}
As operações de $\mathbf{Mon}$--a variedade dos monóides--são $\Omega=\{e,\star\}$ com uma constante, $e$, e um termo binário, $\star$. Sendo $X=\{a,b,c\}$ as leis de monóides são: $$E=\left\{a\star e \approx a, e\star a \approx a, (a\star b)\star c \approx a\star (b\star c)\right\}.$$ $\mathbf{Mon}$ não é protomodular.
\end{cex}

\section{Álgebras topológicas protomodulares}
No caso de $\Top^\Teoria$ extrai-se o resultado seguinte como corolário ao Teorema \ref{this} que irá ser uma base de apoio, como em \cite{Borceux2}.

\begin{cor}
Seja $A$ uma álgebra sobre uma teoria protomodular $\Teoria$ e sejam $e_i$, $\alpha_i$ e $\theta$ como no Teorema \ref{this}. Para cada elemento $a\in A$ existem aplicações contínuas,
\begin{align}
\nonumber
\iota_a:& A \rightarrowtail A^n,& x\mapsto (\alpha_1(x,a),\ldots,\alpha_n(x,a)),\\
\nonumber
\theta_a:& A^n \to A,& (a_1,\ldots, a_n) \mapsto \theta(a_1,\ldots,a_n,a) 
\end{align}
tais que $\theta_a \circ \iota_a = \id_A$ e $\iota_a(a)=(e_1,\ldots, e_n)\in A^n.$
\end{cor}

\begin{proof}
Directa.
\end{proof}

\begin{lem}
Dado um elemento $a$ de uma álgebra sobre uma teoria protomodular, os conjuntos $$\bigcap_{i=1}^n \alpha_i(-,a)^{-1}(U_i),$$ onde os $U_i$ são as vizinhanças abertas das constantes $e_i$, constituem um sistema fundamental de vizinhanças abertas de $a,$ e consequentemente se uma propriedade estável para limites que é válida numa vizinhança de cada constante é válida numa vizinhança de cada ponto.
\end{lem}

\begin{proof}
Obtém-se o resultado aplicando $\iota^{-1}$ ao sistema fundamental de vi\-zin\-han\-ças abertas de $(e_1,\ldots,e_n)$ o conjunto $\{U_1\times\cdots \times U_n \vert e_i\in U_i,\ U_i\text{ aberto}\}.$
\end{proof}

\begin{cor}
Se uma álgebra topológica sobre uma teoria protomodular tiver um sistema de vizinhanças numerável de cada constante, então é metrizável.
\end{cor}

\begin{pro}
Se $\Teoria$ é uma teoria protomodular, então cada $A\in \Obj (\Top^\Teoria)$ é regular.
\end{pro}
\begin{proof}
Basta verificar regularidade numa vizinhança de cada constante, $e_i$. Seja $V$ uma vizinhança aberta de $e_i$. Dado que $\theta_{e_i}: A^{n+1} \to A$ é contínua, $$\theta_{e_i}(U_1\times \cdots \times U_n\times U) \subseteq V,$$ onde $U_k\in \mathcal{V}_{e_k}$ e $U\in \mathcal{V}_{e_i}.$ 
Mostre-se que $\overline{U}\subseteq V.$ Seja $a\in \overline{U}.$ Como o conjunto $$\bigcap_{k=1}^{n} (\alpha_k(-,a))^{-1} (U_k)$$ é aberto, a intersecção $U \bigcap (\alpha_k(-,a))^{-1} (U_k)$ não é vazia. Seja $b$ um elemento da intersecção, como $\alpha_k(a,b)\in U_k,$ $$a=\theta_{e_i} (\alpha_1(a,b),\ldots,\alpha_n(a,b),b) \in \theta_{e_i} (U_1\times \cdots \times U_n\times U) \subseteq V.$$ 
\end{proof}

\begin{pro}
Sejam $\Teoria$ uma teoria protomodular e $A$ uma álgebra sobre $\Teoria.$ Os conjuntos $\{e_i\}$ são fechados se e só se $A$ for de Hausdorff.
\end{pro}

\begin{proof}
($\Rightarrow$) $\{e_i\}$ fechado $\Rightarrow$ $T_1+$regular $\Rightarrow$ de Hausdorff.
\end{proof}

\begin{tma}
Seja $B\leq A$ uma subálgebra sobre uma teoria protomodular $\Teoria$. Se $B$ for aberta como subespaço de $A$ então $B$ é fechada também.
\end{tma}

\begin{proof}
Seja $a \in A \setminus B.$ Define-se um subconjunto especial de $A$, $$U = \bigcap_{i} \alpha_i(a, -)^{-1}(B).$$ $a$ pertence a $U$ porque para $1\leq i \leq n$, $\alpha_i(a,a)=e_i\in B.$ Se existisse $b\in U\cap B$ então $a=\theta(\alpha_1(a,b),\ldots,\alpha_n(a,b),b))\in B.$
\end{proof}

\begin{tma}
Seja $B\leq A$ uma subálgebra sobre $\Teoria$, então, $\overline{B}$ também é uma subálgebra sobre $\Teoria.$
\end{tma}

\begin{proof}
Seja $\tau$ uma operação de aridade $n,$ em $A.$ Mostra-se que se restringe de modo natural para a operação em $\overline{B},$ $$\tau(\overline{B}^n) = \tau (\overline{B^n}) \subset \overline{\tau(B^n)} \subseteq \overline{B}.$$
\end{proof}

\section{Functores Topológicos}
Nesta secção pretendemos discutir como se constroem os colimites em $\Top^\Teoria,$  most\-ran\-do que é cocompleta. Para alcançar esse objectivo apresenta-se sumariamente a noção de functor topológico, o qual reflecte colimites.

\begin{defc}
Seja $\C$ uma categoria. Uma {\it fonte} em $\C$ é um par da forma \hbox{$(X,(f_i:X\to X_i)_{i\in I})$}. A fonte $(X,(f_i:X\to X_i)_{i\in I})$ diz-se {\it inicial} relativamente ao functor \hbox{$F: \C \to \mathbf{D}$} se, para cada outra fonte $(X',(f'_i:X'\to X_i)_{i\in I})$ e morfismo  \hbox{$g:FX'\to FX$} que torna comutativo (para todo o $i\in I$) o diagrama seguinte,
\[
\begin{tikzpicture}
\matrix (m) [matrix of math nodes, row sep=1.5cm, column sep=1.5cm]{& &FX_i \\ FX' & FX & \\};
\path[ar] (m-2-1) edge node[auto,swap]{$g$} (m-2-2) 
(m-2-2) edge node[auto,swap]{$Ff_i$} (m-1-3);
\path[ar]  (m-2-1) edge[out=50, in=180] node[auto]{$Ff'_i$} (m-1-3);
\end{tikzpicture}
\]
existir um morfismo único $\gamma:X'\to X$ que satisfaça $F\gamma = g$ e que o torne comutativo 
\[
\begin{tikzpicture}
\matrix (m) [matrix of math nodes, row sep=1.5cm, column sep=1.5cm]{& &X_i \\ X' & X & \\};
\path[ar] (m-2-1) edge node[auto,swap]{$\exists!\gamma$} (m-2-2) 
(m-2-2) edge node[auto,swap]{$f_i$} (m-1-3);
\path[ar]  (m-2-1) edge[out=50, in=180] node[auto]{$f'_i$} (m-1-3);
\end{tikzpicture}
\]
para todo o $i\in I.$
\end{defc}

O conceito de functor topológico sugere a noção de aplicação inicial nas aplicações contínuas entre espaços topológicos.

\begin{defc}
Um functor $F:\C\to \mathbf{D}$ diz-se topológico quando, para toda a fonte em $\mathbf{D}$ da forma $$(D, (d_i:D\to FX_i)_{i\in I})$$ para uma família específica $(X_i)_{i\in I}$ de objectos em $\C$, existir uma fonte inicial $(X, (\delta_i:X\to X_i)_{i\in I})$ relativa ao functor $F$,  e um isomorfismo $\lambda:D\to FX$ com $F\delta_i\circ \lambda = d_i,$ como no diagrama:
\[
\begin{tikzpicture}
\matrix (m) [matrix of math nodes, row sep=2cm, column sep=2cm]{& FX_i & \\ FX' & FX & D. \\};
\path[ar] 
(m-2-1) edge node[auto]{$F\delta_i'$} (m-1-2)
(m-2-1) edge node[auto,swap]{$g$} (m-2-2)
(m-2-2) edge node[auto]{$F\delta_i$} (m-1-2)
(m-2-3) edge node[auto]{$\lambda$} node[auto,swap]{$\cong$} (m-2-2)
(m-2-3) edge node[auto,swap]{$d_i$} (m-1-2);
\end{tikzpicture}
\]
\end{defc} 

Apresentamos os resultados suficientes para alcançar o nosso objectivo: mostrar que $\Top^\Teoria$ é cocompleta. Demonstrações completas são omitidas.

\begin{lem}
Seja $F:\C\to \mathbf{D}$ um functor topológico. Se $\mathbf{D}$ for cocompleta, $\C$ igualmente o será.
\end{lem}
\begin{proof}
Este lema encontra-se em \cite{Herrlich} onde é designado por Corolário 6.4.
\end{proof}

\begin{tma}
O functor de esquecimento $U:\Top^\Teoria \to \Conj^\Teoria$ é topológico.
\end{tma}

\begin{proof}
Seja $(S,(g_i:S\to UX_i)_{i\in I})$ uma fonte em $\Conj^\Teoria$ onde $(X_i)_i$ são álgebras topológicas e $S \in \Conj^\Teoria$. Na definição de functor topológico, basta tomar $(X,(\tilde{g}_i:X\to X_i)_{i\in I})$ onde $X$ é $S$ munido da topologia quociente dos $g_i,$ e os $\tilde{g}_i$ são múnidos da estrutura de morfismo de álgebras topológicas.
\end{proof}

\begin{cor}
Seja $\Teoria$ uma teoria algébrica. A variedade de álgebras topológicas $\Top^\Teoria$ é cocompleta e os seus colimites são construidos usando os colimites em $\Conj^\Teoria.$
\end{cor}

É demonstrado em \cite{Wyler} que $\Top^\Teoria$ é cocompleta (e completa) envolvendo functores adjuntos. Todavia, esta abordagem é bastante abstracta e não mostra uma forma directa de calcular os limites em $\Top^\Teoria.$

\section{Operação de Maltsev}

Nas álgebras que pretendemos estudar existem operações particulares que permitem generalizar resultados de $\GrpTop.$ 
\begin{defc}
Uma {\it operação de Maltsev} é uma aplicação $$p: X^3 \to X$$ tal que $p(x,x,y) = y;\ p(y,y,z) = z.$

Uma variedade de álgebras universais diz-se {\it de Maltsev} se existir uma operação de Maltsev nas operações da sua teoria.
\end{defc}
Em \cite{Johnstone} é demonstrado que em álgebras topológicas com uma operação de Maltsev muitos resultados que vimos anteriormente são ainda válidos.

Incluímos um resultado clássico.
\begin{pro}As condições seguintes equivalem-se, para uma variedade $\mathcal{V}:$
\begin{enumerate}
\item $\mathcal{V}$ ser de Maltsev;
\item Toda a relação reflexiva em $\mathcal{V}$ ser de equivalência;
\item $RR'=R'R$ para todas as relações de equivalência sobre cada objecto de $\mathcal{V}$.
\end{enumerate}
\end{pro}

\begin{proof}
(1. $\Rightarrow$ 2.) Sejam $\Teoria$ uma teoria e seja $p$ um termo de Maltsev relacionado com $\Teoria.$  Uma relação $R$ em $\Conj^\Teoria$ é um subobjecto $R\rightarrowtail X\times X$ onde $X$ é uma álgebra sobre $\Teoria.$ Por $R$ ser reflexiva, $(x,x),(y,y)\in R,$ portanto se $(x,y)\in R$ tem-se $(p(x,x,y),p(x,y,y))= (y,x)\in R.$

(2.$\Rightarrow$ 3.) Sejam $R$ e $R'$ relações de equivalência sobre um objecto $X,$ então:
\begin{align*}
(x,z)\in RR' &\Leftrightarrow (z,x) \in R'R\\
&\Leftrightarrow (z,y) \in R,\ (y,x)\in R'\text{ para algum }y\in X\\
&\Leftrightarrow (x,y) \in R',\ (y,z)\in R\text{ para algum }y\in X\\
&\Leftrightarrow (x,z) \in R'R.
\end{align*}

(3.$\Rightarrow$1.) Digamos que $L(S)$ é a álgebra sobre $\Teoria$ livre gerada pelo conjunto de símbolos $S.$ Mostre-se que $\Teoria$ é de Maltsev. Sejam $f:L(x,y,z) \to L(b,c)$ e $g:L(x,y,z) \to L(a,b)$ os morfismos definidos por $f(x)=f(y) = b,\ f(z)=c,$ e $g(x)=a,\ g(y)=g(z)=b.$ Os seus pares núcleos definem uma operação de Maltsev: 
\begin{align*}
(x,y)\in R[f],\ (y,z)\in R[g] &\Leftrightarrow (x,z)\in R[g]R[f]\\
& \Leftrightarrow \exists m \ (x,m)\in R[g],\ (m,y)\in R[f].
\end{align*}
Identificando $m=p(x,y,z)$ provém das propriedade de $f$ e $g$ que $p(x,y,z)$ é de Maltsev.
\end{proof}

Generalizando para quaisquer categorias:
\begin{defc}
Uma categoria é de Maltsev se for finitamente completa e cada relação reflexiva for de equivalência. 
\end{defc}





\section{Quocientes}
Começamos por provar um teorema de Maltsev (em \cite{Maltsev}).
\begin{pro}
\label{l0}
Seja $\Teoria$ uma teoria que contenha uma operação de Maltsev entre os termos de $\Teoria.$ Se $q:X\twoheadrightarrow Q$ for um epimorfismo regular em $\Top^\Teoria$ é um morfismo aberto.
\end{pro}

\begin{proof}
Seja $U$ um aberto e considere-se $x\in q^{-1}q(U).$ Existe $u\in U$ com $q(u)=q(x),$ e como $p(x,x,u)=x$ e $p:X^3\to X$ é contínua (e $p(-,x,u)$ é em particular uma aplicação contínua entre espaços topológicos) existe uma vizinhança $V\in \mathcal{V}_x$ tal que $p(V,x,u)\subset U.$ Verifica-se facilmente que $V\subseteq q^{-1}(q(U)),$ dando o pretendido.
\end{proof}

\begin{tma}
Quando $\Teoria$ for protomodular, os epimorfismos regulares em $\Top^\Teoria$ serão exactamente os morfismos sobrejectivos e abertos.
\label{l1}
\end{tma}

\begin{proof}
Como $\Teoria$ admite uma operação de Maltsev, a proposição anterior diz-nos que os epimorfismos regulares de $\Top^\Teoria$ são abertos. Através do modo de construção, é visível que um co-igualizador topológico é sobrejectivo.
\[
\begin{tikzpicture}
\matrix (m) [matrix of math nodes, column sep=1.2cm,row sep=1.2cm,text height=1.5ex, text depth=0.25ex]{R[f] & X & Y\\&&Z.\\};
\path[ar]
([yshift=.5mm]m-1-1.east) edge node[auto]{$\pi_1$} ([yshift=.5mm]m-1-2.west)
([yshift=-.5mm]m-1-1.east) edge node[auto,swap]{$\pi_2$} ([yshift=-.5mm]m-1-2.west)
(m-1-2) edge node[auto]{$f$} (m-1-3)
(m-1-2) edge node[auto,swap]{$z$} (m-2-3);
\path[ar,dashed] (m-1-3) edge node[auto]{$g$} (m-2-3);
\end{tikzpicture}
\]

A necessidade da condição prova-se mostrando que um morfismo aberto e sobrejectivo $f: X\to Y$ é o co-igualizador do seu par núcleo. Se $z:X\to Z$ for outro morfismo com $\pi_1\circ f= \pi_2\circ f$ basta provar-se que a função entre os conjuntos subjacentes é contínua. Essa afirmação é válida porque para cada aberto $U$ de $Z,$ $f^{-1}(g^{-1}(U)) = z^{-1}(U)$ é aberto, logo, porque $f$ é aberto e sobrejectivo $g^{-1}(U) = f(f^{-1}(g^{-1}(U)))$ é aberto.
\end{proof}

\begin{tma}
\label{l2}
A variedade $\Top^\Teoria$ é regular para toda a teoria protomodular $\Teoria.$
\end{tma}
\begin{proof}
$\Top^\Teoria$ tem todos os co-igualizadores e um co-igualizador é uma aplicação que é contínua, aberta e sobrejectiva. Seja $f:X\to Z$ um epimorfismo regular e seja o quadrado a seguir um produto fibrado:
\[
\begin{tikzpicture}
\matrix (m) [matrix of math nodes, column sep=1.5cm, row sep= 1.5cm,text height=1.5ex, text depth=0.25ex] {X\times_Z Y & Y \\ X & Z.\\};
\path[ar] (m-1-1) edge node[auto]{$\pi_2$}(m-1-2)
(m-1-1) edge node[auto,swap]{$\pi_1$}(m-2-1)
(m-1-2) edge node[auto]{$g$}(m-2-2)
(m-2-1) edge node[auto,swap]{$f$}(m-2-2);
\end{tikzpicture}
\]
Mostre-se que $\pi_2$ é um epimorfismo regular também. Basta fazer isso na categoria $\Top.$ Como os epimorfismos são estáveis para produtos fibrados, $\pi_2$ é sobrejectivo. Uma base da topologia de $X\times_Z Y$ são os conjuntos $U=\pi_1^{-1}(A) \cap \pi_2^{-1}(B),$ onde $A\subseteq X$ e $B\subseteq Z$ são abertos. O conjunto $$\pi_2(U) = \{y\in B | \exists x\in A: f(x)=g(y)\} = B\cap g^{-1}(f(A))$$ é um aberto ($f$ é uma aplicação aberta). Como $\pi_2(U)$ é aberto para todos os elementos de uma base de $X\times_Z Y$, $\pi_2$ é uma aplicação aberta.

\end{proof}

\chapter{Produtos semidirectos}
O objectivo deste capítulo centra-se na caracterização dos produtos semidirectos em variedades das álgebras topológicas. Recorda-se primeiramente a noção de produto semidirecto em $\Grp$ e extrai-se as suas propriedades. Em seguida, introduz-se a noção de functor monádico, que é utilizado para definir os produtos semidirectos em categorias como em \cite{Bourn3}. Conclui-se averiguando que um resultado em \cite{Inyangala} para $\Alg(\Teoria)$ é ainda válido para álgebras topológicas.

\section{Produtos semidirectos em $\Grp$}
Recorde-se a definição do produto directo em $\Grp:$ Um grupo $G$ diz-se um produto directo se tiver dois subgrupos normais: $N,N'\lhd G,$ tal que $N \cap N' = \{e\}$ e $N N'= G.$ O conceito do produto semidirecto provém de um relaxamento desta definição--requerendo somente que um dos subgrupos seja normal. 

\begin{defc}
Um grupo $G$ diz-se um {\it produto semidirecto} se tiver dois subgrupos $K,Q\leq G$ com um subgrupo normal $K$ que se {\it complementam}, i.e., $K\cap Q=\{e\}$ e $KQ=G.$
\end{defc}

\begin{tma}
O grupo $G$ é o produto semidirecto de $K$ por $Q$ se e só se existir a sucessão exacta curta cindida em $\Grp$:
\begin{equation} 
\begin{tikzpicture}[text height=1.5ex,text depth=.25ex]
\matrix (m) [matrix of math nodes, column sep=1.7cm]{K & G & Q,\\};
\path[ar] ([yshift=-.5mm]m-1-3.west) edge node[auto]{$s$} ([yshift=-.5mm]m-1-2.east);
\path[nar] (m-1-1) edge node[auto]{$k$} (m-1-2);
\path[rar] ([yshift=.5mm]m-1-2.east) edge node[auto]{$q$} ([yshift=.5mm]m-1-3.west);
\end{tikzpicture}\label{thisone}
\end{equation}
\end{tma}

\begin{proof}
Supondo que $k$ e $s$ sejam inclusões de grupos, e que $G$ é o produto semidirecto de $K$ por $Q$ prova-se que \ref{thisone} é exacta. Por definição, $K$ é um núcleo, pelo que basta mostrar que $Q\cong G/K$. Seja $\pi: G \to G/K$ a projecção canónica, e seja $\varpi:Q\to G/K$ a restrição $\pi|_Q$. Como $K\cap Q=\{e\},$ $\varpi$ é injectiva. De facto, $\varpi$ é um isomorfismo, pois é sobrejectivo: se $g\in G$ então $g=ab$ onde $a\in K$ e $b\in Q$, logo $G/K\ni gK= abK=bK\in \im (\varpi).$

Como $\varpi$ é um isomorfismo, defina-se $s$ como $\varpi^{-1}$ cujo codomínio é agora estendido a $G$ (supondo ainda que $Q\leq G).$ Verifica-se, então, que $s$ satisfaz $q\circ s = 1_Q.$

$(\Leftarrow)$ Seja $x\in K\cap Q.$ Como $q\circ k(x) = e,$ $x$ é o elemento neutro. Um elemento $g\in G$ tem sempre a factorização $$g=(s\circ q(g)) \cdot (s\circ q (g^{-1})g).$$ O primeiro factor, $s\circ q (g)$ pertence a $Q$, e o segundo pertence a $K$: $q(s\circ q(g^{-1})g) = q(g^{-1} g) = e.$ Esta factorização é única, pois se $g=a_1b_1=a_2b_2$ com $a_1,a_2\in K$ e $b_1,b_2 \in Q,$ então $a_1^{-1}a_2=b_2^{-1}b_1 \in K\cap Q = \{e\}.$
\end{proof}

\begin{pro}
Dado grupos $K$ e $Q$ e uma acção $\varphi:Q\to \mathrm{Aut}(K)$ quaisquer, existe o produto semidirecto de $K$ por $Q$ tal que para $a\in K$ e $b\in Q$, $ba=\varphi_b(a) b$, que é único a menos de isomorfismo.
\end{pro}

\begin{proof}
Para $b\in Q$ designe-se $\varphi(b):K\to K$ por $\varphi_b.$ 

Seja $K\rtimes_\varphi Q$ o conjunto $|K|\times |Q|$ munido da operação de grupo $(a,b)\cdot (a',b') = (a\varphi_b(a'),bb').$ Mostre-se que $Q\rtimes_\varphi K$ é um grupo. 
\begin{itemize}
\item Identidade: $(1,1)$ satisfaz $(a,b)(1,1) = (a,b) = (1,1)(a,b).$
\item Associatividade: $$\left( (a,b)(a',b') \right) (a'',b'') = (a\varphi_b(a')\varphi_b(\varphi_{b'}(a'')),bb'b'') =  (a,b) \left((a',b')  (a'',b'')\right)$$
\item Invertibilidade: $$(a,b) (\varphi_{b^{-1}} (a^{-1}),b^{-1}) = (1,1) = (\varphi_{b^{-1}} (a^{-1}),b^{-1})  (a,b).$$
\end{itemize}

Como $K \rightarrow K \rtimes_\varphi Q;\ x\mapsto (1,x)$ e $Q \rightarrow K \rtimes_\varphi Q;\ x\mapsto (x,1)$ são homomorfismos injectivos pode pensar-se que $K$ e $Q$ são subgrupos de $K \rtimes_\varphi Q.$ Decorre da definição que eles são subgrupos complementares.

Mostre-se que $K\lhd K \rtimes_\varphi Q:$
$$(1,b) (a,1) (1,b)^{-1} = (\varphi_b(a),1).$$

A unicidade desta contrução é uma consequência da factorização de cada elemento num elemento de $K$ e de $Q$.
\end{proof}

\begin{defc}
Seja $(G,\cdot)$ um grupo, um outro grupo $(X,+)$ é chamado um {\it $G$-grupo} se tiver uma acção $\varphi:G\to \mathrm{Aut}(X).$

Estes grupos e os seus homomorfismos, $\rho:X\to Y$, (que além de serem homomorfismos de grupos devem satisfazer $\rho(\varphi_g^X(x))=\varphi_g^Y (\rho(x)),\ \forall g\in G,\ x\in X$), formam uma categoria, a qual se designa $\Grp (G).$
\end{defc}

Cada objecto de $\Grp(G)$ define um ponto sobre $G$: 
\[
\begin{tikzpicture}[text height=1.5ex,text depth=.25ex]
\matrix (m) [matrix of math nodes, column sep=.8cm, row sep=.6cm] 
{X\rtimes G & G.\\};
\path[ar]
([yshift=.5mm]m-1-1.east) edge node[auto]{$p$} ([yshift=.5mm]m-1-2.west)
([yshift=-.5mm]m-1-2.west) edge node[auto]{$s$} ([yshift=-.5mm]m-1-1.east);
\end{tikzpicture}
\]

Além disso, um homomorfismo desses grupos, $\rho:X\to Y$ corresponde a um morfismo de pontos sobre $G:$ 

$$\tilde{\rho}: X\rtimes G \to Y \rtimes G;\ \ (x,g) \mapsto (\rho(x),g).$$

Como pode verificar-se, esse procedimento é functorial.

\begin{tma}
Existe uma equivalência entre as categorias $\Grp (G)$ e $\mathrm{\mathbf{Pt}}_G (\Grp).$
\end{tma}

\begin{proof}
Nesta demonstração é subentendido que a operação binária é escrita aditivamente. 

Seja $\Phi:\Grp(G) \to \mathbf{Pt}_G(\Grp)$ anteriormente definido. É um functor: para $\rho:X\to Y$ e $\sigma:Y\to Z$: $$\Phi(\sigma\circ\rho) = \left((X \rtimes G\rightleftarrows G) \mathop{\to}\limits^{\tilde{\rho}} (Y \rtimes G\rightleftarrows G )\mathop{\to}\limits^{\tilde{\sigma}} (Z \rtimes G \rightleftarrows G)\right)= \Phi(X) \to \Phi(Z),$$ e $\Phi(1_X) = 1_{X\rtimes G \rightleftarrows G}.$

Por outro lado, existe um functor que atribui a cada ponto sobre $G$ um objecto de $\Grp(G):$

$$\Gamma: (p,s:X\rightleftarrows G) \mapsto K\rtimes_\varphi G,$$ onde $\varphi_g(x)$ é o elemento único, $x'\in K$ com $\ker p (x')=s(g)+\ker p(x) -s(g)$ e $K=K[p].$

Confirma-se intuitivamente que, no cálculo $\Gamma\Phi(X,\varphi:G\to\Aut(X))=(X,\tilde{\varphi})$, $\tilde\varphi_g(x)=\varphi_g(x),$ para todo o $g\in G,$ $x\in X.$

Inversamente, confirme-se que $\Phi\Gamma(p,s:X\rightleftarrows G)$ é o diagrama: 
\[
\begin{tikzpicture}
\matrix (m) [matrix of math nodes, row sep=1.7cm, column sep=1.7cm,text height=1.5ex,text depth=.25ex]{K\rtimes_{\tilde\varphi} G & G\\};
\path[ar]
([yshift=.5mm]m-1-1.east) edge node[auto]{$\pi_G$} ([yshift=.5mm]m-1-2.west)
([yshift=-.5mm]m-1-2.west) edge node[auto]{$\iota_G$} ([yshift=-.5mm]m-1-1.east);
\end{tikzpicture}
\] onde $K=K[p].$ Confirme-se que $(p,s)$ e $(\pi_G,\iota_G)$ são objectos isomorfos de $\mathbf{Pt}_B(\Grp).$

Começa-se por definir um morfismo $u:X\to K \rtimes_{\tilde\varphi} G;\ u(x) = (x_0,p(x)),$ onde $x_0$ é o único elemento de $K$ que verifica $\kappa(x_0) = x-\iota_G\pi_G(x).$ A função $u$ torna o diagrama seguinte comutativo:

\[
\begin{tikzpicture}
\matrix (m) [matrix of math nodes, row sep=1.7cm, column sep=1.7cm,text height=1.5ex,text depth=.25ex]{K & X & G\\K & K\rtimes_\psi G & G\\};
\path[ar] (m-1-1) edge node[auto]{$\kappa$} (m-1-2)
(m-2-1) edge node[auto,swap]{$i$} (m-2-2);
\path[-,double distance=.5mm] (m-1-1) edge[double] (m-2-1)
(m-1-3) edge[double] (m-2-3);
\path[ar] 
(m-1-2) edge node[auto]{$u$} (m-2-2)

([yshift=.5mm]m-2-2.east) edge node[auto]{$\pi_G$} ([yshift=.5mm]m-2-3.west)
([yshift=-.5mm]m-2-3.west) edge node[auto]{$\iota_G$} ([yshift=-.5mm]m-2-2.east)

([yshift=.5mm]m-1-2.east) edge node[auto]{$p$} ([yshift=.5mm]m-1-3.west)
([yshift=-.5mm]m-1-3.west) edge node[auto]{$s$} ([yshift=-.5mm]m-1-2.east);

\end{tikzpicture}
\]

Mostre-se que $u$ é um homomorfismo entre grupos. Tem-se $$u(x+y) = (p(x+y),z_0),$$ com $\kappa(z_0)=x+y-s\circ p(x)-s\circ p(y).$ Por outro lado, tem-se $$u(x)+u(y)= (p(x),x_0)+(p(y),y_0) = (p(x)+p(y),x_0+\tilde\varphi_{p(x)}y_0)$$ portanto, $u$ é um homomorfismo entre grupos se $z_0=x_0+\tilde\varphi_{p(x)}y_0,$ ou seja, se se verifica $\kappa(z_0)=\kappa(x_0)+\kappa(\tilde\varphi_{p(x)}y_0),$ que é um cálculo directo: $$\underbrace{x+y-sp(y)-sp(x)}_{\kappa(z_0)}=\underbrace{x-sp(x)}_{\kappa(x_0)} + \underbrace{s(x)+\kappa(y_0)-sp(x)}_{\kappa(\tilde\varphi_{p(x)}y_0)}.$$

Como $u$ é um homorfismo entre grupos e $s\circ\pi_G$ é o seu inverso, $u$ é um isomorfismo de grupos.
\end{proof}

\section{Mónadas e functores monádicos}

Apresentam-se brevemente as definições e os resultados sobre functores monádicos, os quais serão proveitosos para definir produtos semidirectos, na sua globalidade.

\begin{defc}
Uma {\it mónada} $\mathbf{T}$ é um terno $(T,\eta,\mu),$ onde $T:\C\to\C$ é um endofunctor, e $\eta:T\Rightarrow 1_\C,$ e $\mu:T^2\Rightarrow T$ são transformações naturais, que satisfazem os diagramas que se seguem:
\begin{multicols}{2}
\[\begin{tikzpicture}
\matrix (m) [matrix of math nodes, column sep=1.5cm, row sep=1.5cm]{T^3 & T^2\\T^2 & T\\};
\path[thin, ar, thin, double distance=.5mm]
(m-1-1) edge[double] node[auto]{$T\mu$} (m-1-2)
(m-1-1) edge[double] node[auto,swap]{$\mu T$} (m-2-1)
(m-1-2) edge[double] node[auto]{$\mu$} (m-2-2)
(m-2-1) edge[double] node[auto,swap]{$\mu$} (m-2-2);
\end{tikzpicture}\]
\[\begin{tikzpicture}
\matrix (m) [matrix of math nodes, column sep=1.5cm, row sep=1.5cm]{T & T^2 & T\\ & T &. \\};
\path[thin, ar, thin, double distance=.5mm]
(m-1-1) edge[double] node[auto]{$\eta T$} (m-1-2)
(m-1-3) edge[double] node[auto,swap]{$ T\eta$} (m-1-2)
(m-1-2) edge[double] node[auto]{$\mu$} (m-2-2);
\draw (m-1-1) edge[double,double distance=0.5mm] (m-2-2)
(m-1-3) edge[double,double distance=0.5mm] (m-2-2);
\end{tikzpicture}\]
\end{multicols}
\end{defc}

\begin{defc}
Seja $\mathbf{T}=(T,\eta,\mu)$ uma mónada numa categoria $\C.$ Uma {\it álgebra de tipo} $\mathbf{T}$ é um diagrama: 
\[
\begin{tikzpicture}
\matrix (m) [matrix of math nodes,column sep=1.5cm, row sep=1.5cm]{T(X) & X \\};
\path[ar] (m-1-1) edge node[auto]{$\zeta$} (m-1-2);
\end{tikzpicture}
\]
designado simplesmente por $(X,\zeta)$ e tal que
\[
\begin{tikzpicture}
\matrix (m) [matrix of math nodes,column sep=1.5cm, row sep=1.5cm]{T^2X & T X & X \\ TX & X & .\\};
\path[ar] (m-1-1) edge node[auto]{$\mu_X$} (m-1-2)
(m-1-1) edge node[auto,swap]{$T\zeta$} (m-2-1)
(m-2-1) edge node[auto,swap]{$\zeta$} (m-2-2)
(m-1-2) edge node[auto,swap]{$\zeta$} (m-2-2)
(m-1-3) edge node[auto,swap]{$\eta_X$} (m-1-2);
\draw[double distance=.5mm] (m-1-3) edge[double] (m-2-2);
\end{tikzpicture}
\]

Um morfismo entre álgebras $(X,\zeta)\to (Y,\xi)$ é um morfismo $\Xi:X\to Y$ tal que:
\[
\begin{tikzpicture}
\matrix (m) [matrix of math nodes,column sep=1.5cm, row sep=1.5cm]{TX & T Y\\ X& Y.\\};
\path[ar] (m-1-1) edge node[auto]{$T\Xi$} (m-1-2)
(m-1-1) edge node[auto,swap]{$\zeta$} (m-2-1)
(m-2-1) edge node[auto,swap]{$\Xi$} (m-2-2)
(m-1-2) edge node[auto]{$\xi$} (m-2-2);
\end{tikzpicture}
\]
Estas álgebras e os seus respectivos morfismos formam uma categoria que se designa $\C^\mathbf{T},$ que é por sua vez chamada de {\it categoria de Eilenberg-Moore.}
\end{defc}

\begin{tma}
Qualquer mónada surge de um par de functores adjuntos e vice-versa.\label{tmamon}
\end{tma}

\begin{nota}
Note-se que, dada uma mónada, o teorema acima diz que se pode construir uma adjunção que define essa mesma mónada, no entanto, o teorema não diz nada acerca da unicidade de uma adjunção com esta propriedade. De facto, pode haver várias adjunções que definem a mesma mónada.
\end{nota}

\begin{proof}
Seja $(T,\eta,\mu)$ uma mónada.
Definam-se os functores: $$U^\mathbf{T}: \C^\mathbf{T}\to \C;\ U^\mathbf{T}(A,\alpha) = A$$ e $$F^\mathbf{T}: \C \to  \C^\mathbf{T};\ F^\mathbf{T}C = (TC,\mu_C).$$ Confirma-se que $T=U^\mathbf{T}\circ F^\mathbf{T}$ e que são functores adjuntos.

Por outro lado, se $(F,U,\mu,\epsilon)$ é uma adjunção, atendendo à definição de mónada, confirma-se que $\mathbf{T}=(UF,\eta,U\epsilon F)$ é uma mónada. 
\end{proof}

Considere-se a «diferença» entre uma adjunção $F\dashv U,$ e a adjunção induzida pela categoria de Eilenberg-Moore da mónada: $\mathbf{T}=(U\circ F,\epsilon,\mu)$. Especificamente, podemos questionar quais são as condições sobre os functores adjuntos em que $\mathbf{A}^\mathbf{T} \cong \mathbf{B}.$ Para isso define-se o functor de «comparação» que se segue.

\begin{defc}
Sejam
\[
\begin{tikzpicture}
\matrix(m)[matrix of math nodes,column sep=1.5cm]{\C &|[yshift=-1.75mm]| ^{^\perp} &\mathbf{D}\\};
\path[ar]
([yshift=1.1mm]m-1-1.east) edge node[auto]{$F$} ([yshift=1.1mm]m-1-3.west);
\path[ar]
([yshift=-1.1mm]m-1-3.west) edge node[auto]{$U$} ([yshift=-1.1mm]m-1-1.east);
\end{tikzpicture}
\]
uma adjunção e $\C^\mathbf{T}$ a categoria de Eilenberg-Moore associada a esta adjunção. Ao functor $J:\mathbf{D}\to \C^\mathbf{T}$ que leva cada $D$ na álgebra $(U(D),U(\epsilon_D))$, e cada morfismo $f$ na sua imagem para $U$ chama-se {\it functor de comparação}.
\end{defc}

Veja-se o diagrama comutativo de functores:
\[
\begin{tikzpicture}
\matrix(m)[matrix of math nodes, column sep=2.5cm, row sep=2.5cm]{\mathbf{D} & & \C^\mathbf{T}\\
&\C&.\\};
\path[ar]
(m-1-1) edge node [auto] {$J$} (m-1-3)
([shift={(45:.5mm)}]m-1-1.south east) edge node [auto] {$U$} ([shift={(45:.5mm)}]m-2-2.north west)
([shift={(45:-.5mm)}]m-2-2.north west) edge node [auto] {$F$} ([shift={(45:-.5mm)}]m-1-1.south east)

([shift={(-45:.5mm)}]m-1-3.south west) edge node [auto,swap] {$F^\mathbf{T}$} ([shift={(-45:.5mm)}]m-2-2.north east)
([shift={(-45:-.5mm)}]m-2-2.north east) edge node [auto,swap] {$U^\mathbf{T}$} ([shift={(-45:-.5mm)}]m-1-3.south west);
\end{tikzpicture}
\]

\begin{defc}
Um {\it par reflexivo} é um par de morfismos $f,g:X\to Y$ tal que exista $s:Y\to X$ com $f\circ s = 1_Y = g\circ s.$

\end{defc}

Anuncia-se o célebre critério de Beck cuja demonstração se encontra na dissertação de Jonathan Beck \cite{Beck}.

\begin{tma}
Com a mesma notação acima, o functor de comparação é uma equivalência entre categorias se e só se
\begin{enumerate}
\item $U$ reflectir isomorfismos;
\item $\C$ tiver co-igualizadores de pares reflexivos;
\item $U$ preservar esses.
\end{enumerate}
\end{tma}

\begin{defc}
Um functor $Q:\C \to \mathbf{D}$ chamar-se-á {\it monádico} se tiver um adjunto à esquerda e se o functor de comparação associado for uma equivalência de categorias.
\end{defc}

Como exemplo inclui-se o functor $0_G^*$ cuja demonstração se encontra em \cite{Bourn3}.
\begin{tma}
O functor (\ref{letsusethis}), $\ker_G:\mathbf{Pt}_G(\Grp) \to \Grp$, é monádico. 
\end{tma}

Em \cite{Bourn3}, Dominique Bourn e George Janelidze definem produtos semidirectos num contexto generalizado e mostram que em $\Grp$ a sua construção é idêntica ao produto semidirecto clássico.

\begin{defc}
Seja $\C$ uma categoria pontuada finitamente completa e cocompleta.
\begin{enumerate}
\item Diz-se que $\C$ tem produtos semidirectos se, para todo o objecto $B$ de $\C,$ o functor de esquecimento $$\ker_B=U^B:\Pt{B}\to\C$$ for monádico.
\item Chama-se à mónada correspondente ao functor $U^B$ por $\mathbf{T}^B,$ cujas álgebras designar-se-ão álgebras de tipo $B.$
\item O produto semidirecto $(X,\xi)\rtimes B$ é o diagrama correspondente em $\Pt{B}$ de \[ 
\begin{tikzpicture} 
\matrix (m) [matrix of math nodes,column sep=1.5cm]{ T^B(X) & X.\\}; \path[ar] (m-1-1) edge node[auto]{$\xi$}(m-1-2); 
\end{tikzpicture} \]
\end{enumerate}
\end{defc}

\begin{ex}
Em \cite{Bourn3} é demonstrado que em $\Grp$ essa noção de produto semidirecto é equivalente à noção que se apresenta no início deste capítulo. Os grupoides também são um exemplo de uma categoria com produtos semidirectos como é dado em \cite{Metere}.
\end{ex}

\section{Produtos semidirectos em álgebras topológicas}
Em \cite{Borceux2} é mostrado que para $\Teoria$ semi-abeliana $\Top^\Teoria$ possui produtos semidirectos. Caracterizemo-los num caso particular.

\begin{lem}
Se $\C$ é uma categoria pontuada e finitamente completa e cocompleta para cada objecto $B$ de $\C$ os functores
\begin{align*}
F^B: \C &\to \mathbf{Pt}_B(\C) &\hspace{1cm} U^B: \hspace{.5cm} \mathbf{Pt}_B(\C) &\to \C\\
X &\mapsto ([0,1],\iota_B:X+B\rightleftarrows B) &\hspace{1cm} (p,s:A\rightleftarrows B) &\mapsto K[p],
\end{align*}
são adjuntos, com $F^B \dashv U^B.$
\end{lem}

\begin{proof}
Obtém-se por cálculo directo:
$$U^BF^B(X)= K[X+B \mathop{\longrightarrow}^{[0,1]} B] $$ que se designa $B\flat X.$ Define-se a transformação natural $\eta: \id_\C \Rightarrow U^BF^B$ dos morfismos únicos nos diagramas da forma:
\[
\begin{tikzpicture}
\matrix (m) [matrix of math nodes, column sep=1cm, row sep=1cm,text height=1.5ex,text depth=.25ex]{B\flat X & X+B & B\\ B\flat X & X & B.\\};
\path[ar]
(m-1-1) edge node[auto]{$\kappa_X$} (m-1-2) 
(m-2-2) edge node[auto]{$\iota_X$} (m-1-2)
(m-2-2) edge node[auto,swap]{$0$} (m-2-3);
\path[ar,dashed] (m-2-2) edge node[auto]{$\eta_X$}(m-2-1);
\path[-,double distance=.5mm] (m-1-1) edge[double](m-2-1)
(m-1-3) edge[double](m-2-3);
\path[ar]
([yshift=.5mm]m-1-2.east) edge node[auto]{$[0,1]$} ([yshift=.5mm]m-1-3.west)
([yshift=-.5mm]m-1-3.west) edge node[auto]{$\iota_B$} ([yshift=-.5mm]m-1-2.east);
\end{tikzpicture}
\]
Mostre-se que a propriedade universal é válida:

Seja $u:X\to K[p]$ um morfismo em $\C.$ Prove-se que existe um único $f:X+B\to A$ que satisfaz a propriedade universal para $u.$

O morfismo $U^B(f)$ é o único que torna o diagrama seguinte comutativo:
\[
\begin{tikzpicture}
\matrix (m) [matrix of math nodes, column sep=1cm, row sep=1cm,text height=1.5ex,text depth=.25ex]{B\flat X & X+B\\K[p] & A,\\};
\path[ar] (m-1-1) edge node[auto,swap]{$U^Bf$} (m-2-1)
(m-1-1) edge node[auto]{$\kappa_X$} (m-1-2)
(m-2-1) edge node[auto,swap]{$\ker (p)$} (m-2-2)
(m-1-2) edge node[auto]{$f$} (m-2-2);
\end{tikzpicture}
\]
do qual se simplifica a condição universal: $\ker p$ é um monomorfismo $u= U^B(f)\eta_X$ se e somente se verifica $\ker(p) u =\ker(p) U^B(f)\eta_X = f\kappa_X\eta_X.$

Como $\kappa_X\eta_X=\iota_X,$ obtém-se ainda uma condição universal da adjunção mais simplificada: $\ker(p) u = f\iota_X.\label{c1}$

Ora, a existência de tal $f$ equivale à comutatividade do seguinte:

\[
\begin{tikzpicture}
\matrix (m) [matrix of math nodes, column sep=1.5cm, row sep=1.5cm]{X+B & & A\\ & B &,\\};
\path[ar]
(m-1-1) edge node[auto]{$f$} (m-1-3);
\path[ar]
([shift={(-45:-.5mm)}]m-2-2.north east) edge node[auto]{$s$} ([shift={(-45:-.5mm)}]m-1-3.south west)
([shift={(-45:.5mm)}]m-1-3.south west) edge node[auto]{$p$} ([shift={(-45:.5mm)}]m-2-2.north east);

\path[ar]
([shift={(45:.5mm)}]m-2-2.north west) edge node[auto,swap]{$\iota_B$} ([shift={(45:.5mm)}]m-1-1.south east)

([shift={(45:-.5mm)}]m-1-1.south east) edge node[auto,swap]{$[0,1]$} ([shift={(45:-.5mm)}]m-2-2.north west);
\end{tikzpicture}
\]

isto é, $p f = [0,1]$ e $f\iota_B = s.$

Define-se $f$ como o único morfismo que torna o diagrama comutativo:
\[
\begin{tikzpicture}
\matrix (m) [matrix of math nodes, column sep=1.5cm, row sep=1.5cm,text height=1.5ex,text depth=.25ex]{|[xshift=-4mm]|B & X+B & X\\& A & |[shift={(-45:-5mm)}]|K[p].\\};
\path[ar] (m-1-1) edge node[auto]{$\iota_B$} (m-1-2)
(m-1-1) edge node[auto,swap]{$s$} (m-2-2)
(m-1-3) edge node[auto,swap]{$\iota_X$} (m-1-2)
(m-1-3) edge (m-2-2)
(m-1-3) edge node[auto]{$u$} (m-2-3)
(m-2-3) edge node[auto]{$\ker (p)$} (m-2-2);
\path[ar,dashed] (m-1-2) edge node[auto]{$f$} (m-2-2);
\end{tikzpicture}
\]
De $f\iota_B = s$ e $f\iota_X = \ker(p)u$
obtém-se: $$p (f\iota_B) = ps = 1_B = [0,1]\iota_B,$$ e de modo igual $p f\iota_X = [0,1]\iota_X,$ do qual se conclui que a propriedade universal:

\[
\begin{tikzpicture}
\matrix (m) [matrix of math nodes, column sep=1cm, row sep=1cm,text height=1.5ex,text depth=.25ex]{X & \\ U^BF^B(X) & K[p]\\F^B(X) & (A\mathop{\rightleftarrows}\limits^p_s B)\\};
\path[ar] (m-1-1) edge node[auto,swap]{$\eta_X$} (m-2-1)
(m-1-1) edge node[auto]{$u$} (m-2-2)
(m-2-1) edge node[auto,swap]{$U^Bf$} (m-2-2);
\path[ar,dashed](m-3-1) edge node[auto]{$f$} (m-3-2);
\end{tikzpicture}
\]
é válida.
\end{proof}

Generaliza-se de modo directo o Teorema 2.2.1 de \cite{Inyangala} para álgebras topológicas. Este teorema revela que, para variedades semi-abelianas que satisfazem certos axiomas (que são designadas $\Omega$-{\it loops} em \cite{Inyangala}), os produtos semidirectos são uma generalização directa de produtos semidirectos em $\Grp.$

\begin{tma}
Seja $\Teoria$ uma teoria pontuada (com constante $0$) que contenha $+$ e $-$ entre as suas operações binárias que satisfazem os axiomas: 
\begin{align*}
 x+0=x; && 0+x=x; \\
(x+y)-y=x; &&(x-y)+y=x.
\end{align*}

O produto semidirecto em $\Top^\Teoria$ $$(X,\xi)\rtimes B$$ onde $B$ é um objecto de $\Top^\Teoria$ e $(X,\xi)$ é uma álgebra sobre $T^B$ é o espaço topológico $X\times B$ munido da estrutura \begin{equation*}\omega((x_1,b_1),\ldots,(x_n,b_n)) = (\xi(\omega(x_1+b_1,\ldots,x_n+b_n)-\omega(b_1,\ldots,b_n)),\omega(b_1,\ldots,b_n))\label{that}\end{equation*} para cada operação $\omega \in \Omega$ e elementos $b_i,x_i$ de $B$ e $X$, respectivamente.
\end{tma}

\begin{proof}
Considere-se o diagrama em $\Top^\Teoria$:
\[\begin{tikzpicture}
\matrix (m) [matrix of math nodes, row sep=1.7cm, column sep=1.7cm,text height=1.5ex,text depth=.25ex]{X & A& B,\\};
\path[ar]
(m-1-1) edge node[auto]{$\kappa$} (m-1-2)
([yshift=.5mm]m-1-2.east) edge node[auto]{$p$} ([yshift=.5mm]m-1-3.west)
([yshift=-.5mm]m-1-3.west) edge node[auto]{$s$} ([yshift=-.5mm]m-1-2.east);
\end{tikzpicture}
\]onde $\kappa$ é o núcleo de $p$ e $ps=1_B.$ Nesta demonstração está subentendido que os núcleos escritos com $\kappa$ são inclusões de álgebras e, assim, podemos evitar escrevê-los sempre. Definem-se as funções ($p(a-s p(a))=p(a)-p(a)=0$):
\begin{align*}\zeta: X\times B \to A;\hspace{5mm}& (x,b)\to \kappa(x)+s(b),\\
\chi:A\to X\times B;\hspace{5mm} & a\to (\kappa^{-1}(a-s p(a)),p(a)),\end{align*} que são contínuas:

$\zeta$ é uma composição de funções contínuas. Prova-se que $\chi$ é uma aplicação ao seleccionar um aberto $U$ em $X\times B.$ Como $X\times B$ está contido em $A\times B$, existe um aberto $V$ em $A\times B$ com $U= (X\times B)\cap V.$ Como $\beta:A\to X\times B;\ a\mapsto (a-\beta p(a),p(a))$ é contínua $$\chi^{-1}(U)= \chi^{-1}(V) = \kappa^{-1}(\beta^{-1}(V))$$ é aberta.

Tem-se que $\chi$ e $\zeta$ são inversas uma da outra, isto é, $A$ e $X\times B$ são álgebras bijectivas.

O diagrama seguinte é comutativo como se pode verificar:
\[
\begin{tikzpicture}
\matrix (m) [matrix of math nodes, row sep=1.7cm, column sep=1.7cm,text height=1.5ex,text depth=.25ex]{X & X\times B & B\\ X & A& B.\\};
\path[ar] (m-1-1) edge node[auto]{$<1,0>$} (m-1-2)
(m-2-1) edge node[auto]{$\kappa$} (m-2-2);
\path[-,double distance=.5mm] (m-1-1) edge[double] (m-2-1)
(m-1-3) edge[double] (m-2-3);
\path[ar] 
([xshift=-.5mm]m-1-2.south) edge node[auto,swap]{$\zeta$} ([xshift=-.5mm]m-2-2.north)
([xshift=.5mm]m-2-2.north) edge node[auto,swap]{$\chi$} ([xshift=.5mm]m-1-2.south)

([yshift=.5mm]m-2-2.east) edge node[auto]{$p$} ([yshift=.5mm]m-2-3.west)
([yshift=-.5mm]m-2-3.west) edge node[auto]{$s$} ([yshift=-.5mm]m-2-2.east)

([yshift=.5mm]m-1-2.east) edge node[auto]{$\pi_B$} ([yshift=.5mm]m-1-3.west)
([yshift=-.5mm]m-1-3.west) edge node[auto]{$<0,1>$} ([yshift=-.5mm]m-1-2.east);

\end{tikzpicture}
\]

Caracterizem-se as álgebras sobre a mónada $\mathbf{T}^B = (T^B,\eta,\mu)$, onde $T^B = U^B F^B$, $\eta$ é como foi definida no teorema anterior e $\mu$ é definida pela unicidade do morfismo entre as linhas exactas em:
\[
\begin{tikzpicture}
\matrix (m) [matrix of math nodes, row sep=1.7cm, column sep=1.7cm,text height=1.5ex,text depth=.25ex]{0 & B\flat(B\flat X) & (B\flat X)+B & B & 0\\0 & B\flat X & X+B & B & 0.\\};
\path[ar]

(m-1-3) edge node[auto,swap]{$[\kappa_{B,X},\iota_B]$} (m-2-3)

(m-1-1) edge (m-1-2)
(m-2-1) edge (m-2-2)
(m-1-4) edge (m-1-5)
(m-2-4) edge (m-2-5)

(m-1-2) edge node[auto]{$\kappa_{B,B\flat X}$} (m-1-3)
(m-1-3) edge node[auto]{$[0,1]$} (m-1-4)
(m-2-2) edge node[auto,swap]{$\kappa_{B,X}$} (m-2-3)
(m-2-3) edge node[auto,swap]{$[0,1]$} (m-2-4);
\path[ar,dashed] (m-1-2) edge node[auto,swap]{$\mu^B_X$} (m-2-2);
\path[-,double distance=.5mm] (m-1-4) edge[double] (m-2-4);

\end{tikzpicture}
\]
Confirma-se que $\mathbf{T}^B$ é uma mónada. Seja $t$ um termo de aridade $2n$. Da construção de $X+B$ obtém-se que cada elemento do coproduto é da forma $t(x_1,\ldots,x_n,b_1,\ldots,b_n).$ Temos então

\begin{align*}
T^BX = B\flat X &= K[X+B \mathop{\longrightarrow}^{[0,1]} B]\\
&= \{t(x_1,\ldots,x_n,b_1,\ldots,b_n)\ |\ [0,1](t(0,\ldots,0,b_1,\ldots,b_n))=0\}\\
&= \{t(x_1,\ldots,x_n,b_1,\ldots,b_n)\ |\ t(0,\ldots,0,b_1,\ldots,b_n)=0\}.
\end{align*}

O ponto $A\mathop{\rightleftarrows}\limits^p_s B$ corresponde à álgebra $(X,\xi)$ que satisfaz
\[
\begin{tikzpicture}
\matrix (m) [matrix of math nodes, row sep=1.7cm, column sep=1.7cm,text height=1.5ex,text depth=.25ex]{B\flat X & X + B & B\\ X & A& B.\\};
\path[ar] (m-1-1) edge node[auto,swap]{$\xi$} (m-2-1)
(m-1-1) edge node[auto]{$\kappa_{B,X}$} (m-1-2)
(m-2-1) edge node[auto,swap]{$\kappa$} (m-2-2);
\path[-,double distance=.5mm] (m-1-3) edge[double] (m-2-3);
\path[ar] 
(m-1-2) edge node[auto,swap]{$[s,\kappa]$} (m-2-2)

([yshift=.5mm]m-2-2.east) edge node[auto]{$p$} ([yshift=.5mm]m-2-3.west)
([yshift=-.5mm]m-2-3.west) edge node[auto]{$s$} ([yshift=-.5mm]m-2-2.east)

([yshift=.5mm]m-1-2.east) edge node[auto]{$[0,1]$} ([yshift=.5mm]m-1-3.west)
([yshift=-.5mm]m-1-3.west) edge node[auto]{$i_B$} ([yshift=-.5mm]m-1-2.east);

\end{tikzpicture}
\]

Verifique-se a expressão no teorema. Aplicando $\chi\zeta = 1$ vem
\begin{align*}
\omega((x_1,b_1),\ldots (x_n,b_n)) = &\chi(\omega(\zeta(x_1,b_1),\ldots,\zeta(x_n,b_n)))\\
=& \chi(\omega(x_1+s(b_1),\ldots,x_n+s(b_n)))\\
\end{align*}
e agora escrevendo $\omega(\mathbf{y})$ para $\omega(y_1,\ldots,y_n),$ e $f\mathbf{y}$ para $(f(y_1),\ldots,f(y_n))$ e avaliando $\chi$ explicitamente tem-se:

$$\chi(\omega(\mathbf{x}+s\mathbf{b}) = (\omega(\mathbf{x}+s\mathbf{b})-sp\omega(\mathbf{x}+s\mathbf{b}),p\omega(\mathbf{x}+s\mathbf{b})).$$

Obtemos também $p\omega(\mathbf{x}+s\mathbf{b})=\omega(p\mathbf{x}+ps\mathbf{b})=\omega(\mathbf{b}),$ e ainda se verifica\\ $\omega(\mathbf{x}+s\mathbf{b})-s p\omega(\mathbf{b}) = [s,\kappa](\omega(\mathbf{x}+s\mathbf{b})-\omega(\mathbf{b})).$ Esta expressão juntamente com o último diagrama permite-nos fazer as últimas simplificações:

\begin{align*}
\omega(\mathbf{(x,b)})&=
([s,\kappa](\omega(\mathbf{x}+s\mathbf{b})-\omega(\mathbf{b})),\omega(\mathbf{b})) \\
&= (\xi(\omega(\mathbf{x}+s\mathbf{b})-\omega(\mathbf{x}+\mathbf{b})-\omega(\mathbf{x}+s\mathbf{b})-\omega(\mathbf{b})),\omega(\mathbf{b})).
\end{align*}
\end{proof}

Em \cite{Gray} prova-se que esta descrição de produtos semidirectos é válida numa variedade se e só se ela é uma variedade de {\em $\Omega$-loops.} Os resultados de \cite{Gray,Inyangala} levaram a uma descrição mais geral de produtos semidirectos em variedades topológicas envolvendo produtos que generalizam esta (veja-se \cite{Clementino2}).

\begin{cor}
O produto semidirecto de dois grupos topológicos é o produto semidirecto dos grupos subjacentes munido da topologia produto.
\end{cor}


\appendix
\pagestyle{plain}
\chapter{Epimorfismos e o Teorema de Barr-Kock}
Para definir categorias semi-abelianas, não se pode prescindir algumas classes de epimorfimos definidas em seguida.
\begin{defc}
    Seja $f:A\to B$ um morfismo em \C\
    \begin{itemize}
        \item $f$ diz-se um {\emph epimorfismo forte} se, em qualquer quadrado comutativo
\[
\begin{tikzpicture}
\matrix (m) [matrix of math nodes, row sep=1.4cm, column sep=1.4cm] {A & B\\ C & D\\};
\path[ar] (m-1-1) edge node[auto]{$f$}(m-1-2)
(m-1-2) edge[dashed] node[auto,swap]{$t$} (m-2-1)
(m-1-2) edge node[auto]{$h$} (m-2-2)
(m-1-1) edge node[auto,swap]{$g$} (m-2-1)
(m-2-1) edge node[auto,swap]{$m$} (m-2-2);
\end{tikzpicture}
\]
onde $m$ é um monomorfismo, existir um único $t$ que torne o diagrama comutativo.

        \item $f$ diz-se {\emph um epimorfismo regular} se for um co-igualizador de um par de morfismos.
        \item $f$ diz-se {\emph epimorfismo cindido} se tiver inversa à direita.
        \item $f$ diz-se um epimorfismo forte (regular, cindido, resp.) {\emph estável para produtos fibrados} se o seu produto fibrado ao longo de qualquer morfismo for um epimorfismo forte (regular, cindido, respectivamente).
    \end{itemize}
\end{defc}

\begin{pro} Numa categoria \C\ temos as inclusões:
\[\mathrm{epi.\ cindidos\ \subset\ epi.\ regulares\ \subset\ epi.\ fortes\ \subset\ epimorfismos }\]
\end{pro}

\begin{proof}
A demonstração é concisamente ilustrada nas figuras que se seguem.
\begin{multicols}{2}
Um epimorfismo cindido é regular:\\
$
\begin{tikzpicture}
\matrix (m) [matrix of math nodes, column sep=1.5cm, row sep=1cm]{X & X & Y\\ & & Z\\};
\path[ar] 
([yshift=.5mm]m-1-1.east) edge node[auto]{$s\circ f$} ([yshift=.5mm]m-1-2.west)
([yshift=-.5mm]m-1-1.east) edge node[auto,swap]{$1_X$} ([yshift=-.5mm]m-1-2.west)
([yshift=.5mm]m-1-2.east) edge node[auto]{$f$} ([yshift=.5mm]m-1-3.west)
([yshift=-.5mm]m-1-3.west) edge node[auto]{$s$} ([yshift=-.5mm]m-1-2.east)
(m-1-2) edge node[auto,swap]{$z$} (m-2-3);
\path[ar,dashed](m-1-3) edge node[auto]{$z\circ s$} (m-2-3);
\end{tikzpicture}
$
\end{multicols}
Um epimorfismo regular é forte:
\vspace{-5mm}
\begin{multicols}{3}
$
\begin{tikzpicture}
\matrix (m) [matrix of math nodes, column sep=1.4cm, row sep=1.4cm]{\cdot & \cdot & \cdot \\ & \cdot & \cdot \\};
\path[ar] 
([yshift=.5mm]m-1-1.east) edge node[auto]{$r_1$} ([yshift=.5mm]m-1-2.west)
([yshift=-.5mm]m-1-1.east) edge node[auto,swap]{$r_2$} ([yshift=-.5mm]m-1-2.west)
(m-1-2) edge node[auto]{$f$} (m-1-3)
(m-1-2) edge node[auto,swap]{$z$} (m-2-2)
(m-1-3) edge (m-2-3)
(m-2-2) edge node[auto,swap]{$m$} (m-2-3);
\end{tikzpicture}
$

\[\vspace{1.2cm}\hspace{4mm}\Rightarrow\ \  m\circ z \circ r_1 =  m\circ z \circ r_2\ \  \Rightarrow \]

$
\begin{tikzpicture}
\matrix (m) [matrix of math nodes, column sep=1.4cm, row sep=1.4cm]{\cdot & \cdot & \cdot \\ & \cdot & \cdot \\};
\path[ar] 
([yshift=.5mm]m-1-1.east) edge node[auto]{$r_1$} ([yshift=.5mm]m-1-2.west)
([yshift=-.5mm]m-1-1.east) edge node[auto,swap]{$r_2$} ([yshift=-.5mm]m-1-2.west)
(m-1-2) edge node[auto]{$f$} (m-1-3)
(m-1-2) edge node[auto,swap]{$z$} (m-2-2)
(m-1-3) edge (m-2-3)
(m-2-2) edge node[auto,swap]{$m$} (m-2-3);
\path[ar,dashed] (m-1-3) edge (m-2-2);
\end{tikzpicture}
$
\end{multicols}

Um epimorfismo forte é um epimorfismo:\vspace{-5mm}
\begin{multicols}{3}
\[\vspace{1.3cm}u\circ f = v\circ f\]

$
\begin{tikzpicture}
\matrix (m) [matrix of math nodes, column sep=1.5cm, row sep=1.5cm, text height=1.5ex,text depth=.25ex]{\cdot & \cdot & \cdot \\ \cdot & \cdot & \\};
\path[ar] 
([yshift=.5mm]m-1-2.east) edge node[auto]{$u$} ([yshift=.5mm]m-1-3.west)
([yshift=-.5mm]m-1-2.east) edge node[auto,swap]{$v$} ([yshift=-.5mm]m-1-3.west)
(m-1-1) edge node[auto]{$f$} (m-1-2)
(m-1-1) edge node[auto,swap]{$v\circ f$} (m-2-1)
(m-1-2) edge node[auto]{u} (m-2-2)
(m-2-1) edge node[auto,swap]{$1$} (m-2-2);
\path[ar,dashed] (m-1-2) edge node[auto,swap]{$\exists !$} (m-2-1);
\end{tikzpicture}
$
\vspace{1.3cm}
\[\text{Unicidade }\Rightarrow u=v.\]
\end{multicols}
\vspace{-1.5cm}
\end{proof}

De acordo com o pretendido nesta dissertação, foi indispensável a introdução do teorema de Barr-Kock como \cite{Bourn}.

\begin{tma}{Teorema de Barr-Kock}\\
Seja
\[\begin{tikzpicture}[text height=1.5ex,text depth=.25ex]
\matrix (m) [matrix of math nodes, column sep=0.6cm, row sep=0.6cm] {R[f] && X && Y\\
& \node[draw]{2}; & & \node[draw]{1}; &\\
R[f'] && X' && Y'\\};
\path[ar]
(m-1-1) edge node[auto,swap]{$\gamma$} (m-3-1) 
(m-1-3) edge node[auto,swap]{$g$} (m-3-3)
(m-1-5) edge node[auto]{$h$} (m-3-5)
(m-1-3) edge node[auto]{$f$} (m-1-5)
(m-3-3) edge node[auto,swap]{$f'$} (m-3-5)
([yshift=.5mm]m-1-1.east) edge ([yshift=.5mm]m-1-3.west)
([yshift=-.5mm]m-1-1.east) edge ([yshift=-.5mm]m-1-3.west)
([yshift=.5mm]m-3-1.east) edge ([yshift=.5mm]m-3-3.west)
([yshift=-.5mm]m-3-1.east) edge ([yshift=-.5mm]m-3-3.west);
\end{tikzpicture}
\]
um diagrama qualquer numa categoria com produtos fibrados.
    \begin{enumerate}
        \item Se $(g,\gamma)$ é tal que $\fbox{\rm2}$ é um produto fibrado e $f$ é um epimorfismo forte estável para produtos fibrados, então, $\fbox{\rm1}$ é um produto fibrado.
        \item Além disso, se $g$ é um monomorfismo, então, $h$ também é.
    \end{enumerate}
\label{bk}
\end{tma}

\begin{cor}
\label{cor}
Seja \C\ uma categoria finitamente completa. Se $f:X\to Y$ se factoriza como $f=m\circ e,$ onde $e$ é um epimorfismo forte estável para produtos fibrados, então, $R[e] \simeq R[f]$ se e só se $m$ é um monomorfismo.
\end{cor}
\backmatter
\clearpage
\addcontentsline{toc}{chapter}{Índice Remissivo}
\printindex

\end{document}